\documentclass[a4paper,11pt,reqno]{amsart}
\usepackage{amsmath,amssymb,amsthm,graphicx}
\usepackage[utf8]{inputenc} 
\usepackage{array}
\usepackage[utf8]{inputenc}
\usepackage{textcomp} % often helpful
\usepackage{subcaption}
\usepackage{placeins}
\usepackage{float}
\usepackage{verbatim}
\usepackage{hyperref}
\usepackage{xspace}
\usepackage{tikz-cd}
\usepackage{mathtools}
\usepackage{xcolor}
\usepackage{tensor}
\usepackage{paracol}
\usepackage[autostyle=true]{csquotes}
\usepackage[
  backend=biber,
  style=alphabetic,
  maxnames=3,
  doi=true,
  url=false,
  isbn=false,
  eprint=false,
  giveninits=true
]{biblatex}
\usepackage{accents}
\allowdisplaybreaks
\numberwithin{equation}{section}

\DeclareFieldFormat{title}{#1}
\newtheorem{thm}{Theorem}

\newtheorem{theorem}{Theorem}[section]
\newtheorem{lemma}{Lemma}[section]
\newtheorem{definition}{Definition}[section]
\newtheorem{prop}{Proposition}[section]
\newtheorem{remark}{Remark}[section]

\newtheorem{ex}{Example}[section]
\newtheorem{cor}{Corollary}[section]

\DeclareMathOperator{\sign}{sign}
\DeclareMathOperator{\diag}{diag}
\DeclareMathOperator{\Aut}{Aut}
\DeclareMathOperator{\Img}{Im}
\DeclareMathOperator{\Rel}{Re}
\DeclareMathOperator{\End}{End}
\DeclareMathOperator{\tr}{tr}
\DeclareMathOperator{\ad}{ad}
\DeclareMathOperator{\Hom}{Hom}
\DeclareMathOperator{\id}{id}
\DeclareMathOperator{\ind}{ind}
\DeclareMathOperator{\Vol}{Vol}
\DeclareMathOperator{\Diff}{Diff}

\DeclareUnicodeCharacter{2082}{$_2$}
\DeclareUnicodeCharacter{03B7}{$\nu$}

\title{The Crowley--Nordstr\"om invariants of $G_2$-Structures on Aloff--Wallach Spaces}

\author{Artem Aleshin}
\address{Stony Brook University, Stony Brook, NY, USA}
\email{artem.aleshin@stonybrook.edu}

\begin{document}
\begin{abstract}
We study homogeneous $G_2$-structures on the Aloff--Wallach spaces $N_{k,l}=SU(3)/S^1_{k,l},$ and compute the Crowley--Nordstr\"om invariants associated to these structures. In the case where the first Pontryagin class is rationally trivial, we derive a new intrinsic formula for the $\xi$ invariant. Using Goette's expressions for the $\eta$ invariants of homogeneous spaces, we obtain explicit formulas for the invariants of homogeneous $G_2$-structures on $N_{k,l}$. In particular, we prove that
\[
\bar{\nu}(\varphi)=0,
\qquad
\xi(\varphi)=\frac{3}{2}kl(k+l).
\]

As a consequence, we obtain examples of manifolds carrying nearly-parallel $G_2$-structures that lie in distinct connected components of the moduli space of $G_2$-structures.

%We also compare these invariants for the nearly-parallel $G_2$-structures arising from 3-Sasakian structures.
\end{abstract}

\maketitle
\section{Introduction and main results}
\label{intro}
In \cite{Crowley2015} Crowley and Nordstr\"om introduced new invariants $\nu$ and $\xi$ of $G_2$-structures, which were defined in terms of characteristic classes of the coboundary. Later in \cite{Crowley2025} together with Goette they proved an intrinsic formula for the $\nu$ invariant and introduced a $\mathbb{Z}$-valued refinement $\bar{\nu}$ of $\nu$. The invariants $\nu$ and $\xi$ are preserved under the deformations of $G_2$-structures. Although $\bar{\nu}$ is not preserved under arbitrary deformations of $G_2$-structures, it is preserved under deformations through parallel $G_2$-structures and through $G_2$-structures inducing the metrics of positive scalar curvature. In \cite{Crowley2025}, the $\bar{\nu}$ invariant was used to distinguish parallel $G_2$-structures on the extra-twisted connected sums. 

In this paper, we study these invariants in the setting of homogeneous $G_2$-structures, with particular emphasis on nearly-parallel examples. A natural family is provided by the Aloff--Wallach spaces $N_{k,l}$, each of which admits four homogeneous nearly-parallel $G_2$-structures. The problem of computing $\nu$ invariants of these structures was suggested by \cite{Ball2019}, where the authors showed that nearly parallel $G_2$-structure on the same Aloff--Wallach space can be distinguished using $G_2$-instantons.

Moreover, distinct Aloff--Wallach spaces may be diffeomorphic \cite{Kreck1991}. Consequently, a single smooth manifold can admit homogeneous nearly-parallel $G_2$-structures arising from non-equivalent homogeneous structures. This naturally leads to the question of whether these structures are homotopic or not.

In the first part of the paper, we derive an intrinsic formula for the $\xi$ invariant in the special case when $p_1$ is a torsion class. The existence of such a formula and a strategy for deriving it was suggested to us by J. Nordstr\"om. We then compute the $\bar{\nu}$ invariants and $\xi$ invariants of homogeneous nearly-parallel $G_2$-structures on Aloff--Wallach spaces. Our approach uses the formulas for the $\eta$ invariants of homogeneous spaces proven in \cite{Goette2009}, which allow explicit computations of these invariants.
\begin{thm}[An intrinsic formula for the $\xi$ invariant]
	\label{thm:xiintr}
	Let $M^7$ be a $7$-dimensional spin manifold such that its first Pontryagin class is rationally trivial: $p_1(M) = 0 \in H^4(M,\mathbb{Q})$. Let $\varphi$ be a $G_2$-structure on $M$, and let $\zeta$ be the associated unit spinor, then $\xi$ invariant can be computed as 
	\[
		\xi(\varphi) = 14 \int_M \zeta^* \psi + \frac{45}{2} \eta(B_M) - \frac{3}{8} \int_M p_1(\nabla^{LC}) \wedge h(\nabla^{LC}),
	\]
	where $\psi$ is the Mathai-Quillen current, $B_M$ is the odd signature operator, $p_1(\nabla^{LC}$ is the first Pontryagin form with respect to the Levi-Civita connection and $h(\nabla^{LC})$ is such $3$-form that $dh(\nabla^{LC}) = p_1(\nabla^{LC})$.
\end{thm}

Our main result computes the $\bar{\nu}$ and $\xi$ invariants for the homogeneous $G_2$-structures on Aloff--Wallach spaces.
\begin{thm}
	\label{thm:mainres}
	Let $N_{k,l}$ be the Aloff--Wallach space and let $\varphi$ be a homogeneous $G_2$-structure on $N_{k,l}$.
	\begin{enumerate}
		\item If $\varphi$ induces a metric of positive scalar curvature, then 
			\[
				\bar{\nu}(\varphi) = 0.
			\]
		\item If $\varphi$ induces the chosen orientation on the Aloff--Wallach space $N_{k,l}$, then 
			\[
				\nu(\varphi) = 0, \text{ and } \xi(\varphi) = \frac{3}{2}kl(k+l).
			\]
	\end{enumerate}
\end{thm}
Two $G_2$-structures on $M$ are called \textit{deformation equivalent} if they can be connected by a path of $G_2$-structures up to the action of diffeomorphism (detailed definition is given in Section \ref{sect:homcl}).  When the first Pontryagin class is rationally trivial of $M$, $\xi$ is a complete invariant of the deformation equivalence classes. From this we obtain the following corollaries.
\begin{cor}
	\label{cor:noneq}
	Let $(k,l), (k',l')$ be pairs such that the manifolds $N_{k,l}$ and $N_{k',l'}$ are diffeomorphic, but $kl(k+l) \neq \pm k'l'(k'+l')$ (which is equivalent to requiring that the subgroups $S^1_{k,l}$ and $S^1_{k',l'}$ are not conjugate inside $SU(3)$), then any homogeneous structures $\varphi_{k,l}$ and $\varphi_{k',l'}$ are not homotopic.
\end{cor}

\begin{remark}
    The Aloff--Wallach space $N_{1,-1}$ is the only Aloff--Wallach space admitting a self-diffeomorphism reversing orientation. It follows that the homogeneous $G_2$-structures inducing different orientations on $N_{1,-1}$ are deformation equivalent, and hence their invariants $\nu$ and $\xi$ vanish. This is consistent with the formulas proven in this paper, which imply that, up to permutation, $(1,-1,0)$ is the unique triple such that corresponding homogeneous $G_2$-structures have vanishing $\nu$ and $\xi$. 
\end{remark}
Consequently, we obtain examples of manifolds admitting nearly-parallel $G_2$-structures inducing the same orientations but lying in different components of the deformation space of $G_2$-structures. 
\begin{ex}
	These examples are given by the pairs of Aloff--Wallach spaces that were found in \cite{Kreck1991}, two such pairs are 
\[
	N_{-4638861,582656} \cong N_{-2594149,5052965},
\]
and
\[
	N_{8287903,14825547} \cong N_{-18401045,21746838}.
\]
\end{ex}

\begin{remark}
	Aloff--Wallach spaces also produce examples of pairs of closed manifolds admitting nearly-parallel $G_2$-structures that are homeomorphic but not diffeomorphic, similarly to the results from \cite{Crowley2021} for manifolds with parallel $G_2$-structures. This follows from the homeomophism and diffeomorphism classification of Aloff--Wallach spaces given in \cite{Kreck1991}.
\end{remark}

We also compare, for an arbitrary 3-Sasakian manifold, the nearly-parallel $G_2$-structures associated to a 3-Sasakian structure $\varphi_{ts}$ with the squashed nearly parallel $G_2$-structure $\varphi_{sq}$. We note that these structure are in fact homotopic through $G_2$-structures inducing the metrics of positive scalar curvature, and hence

\[
	\bar{\nu}(\varphi_{sq}) = \bar{\nu}(\varphi_{ts}),
\]
and
\[
	\xi(\varphi_{sq}) = \xi(\varphi_{ts}).
\]
In particular, it follows from the $h$-principle \cite[Theorem 1.8]{Crowley2015} that there exists a path of coclosed $G_2$-structures connecting $\varphi_{ts}$ and $\varphi_{sq}$.
%In particular, this shows that neither of invariants that we are considering distinguish between these nearly parallel $G_2$-structures, even though they can be distinguished using finer gauge-theoretic methods such as deformed $G_2$-instantons \cite{Lotay2022}.

As a consequence, we now know all the $\bar{\nu}$ invariants of the homogeneous proper nearly-parallel $G_2$ manifolds. Namely, up to the sign defined by choice of orientation:

\begin{itemize}
	\item By Example \ref{ex:sphere} \[ \bar{\nu}(\varphi_{sq}(S^7)) = 1,\]
	\item By Theorem \ref{thm:mainres} \[ \bar{\nu}(\varphi(N_{k,l})) =  0,\]
	\item By [\cite{Crowley2025}, Example 1.8] \[ \bar{\nu}(\varphi(SO(5)/SO(3))) = 1 .\]
\end{itemize}

Using an intrinsic formula for the $\xi$ invariant we are also able to compute the $\xi$ invariant for the $G_2$-structure on the Berger space $SO(5)/SO(3)$. 

So, we also know the $\xi$ invariant for all of the homogeneous proper nearly-parallel $G_2$-structures:
\begin{itemize}
	\item By \cite[Example 1.14]{Crowley2015} \[ \xi(\varphi_{sq}(S^7)) = -7,\]
	\item By Theorem \ref{thm:mainres} \[ \xi(\varphi(N_{k,l})) = \frac{3}{2}kl(k+l),\]
	\item By Theorem \ref{thm:xiintr} and \cite[Proposition 2.5]{Goette2004} \[ \xi(\varphi(SO(5)/SO(3))) = -\frac{11}{10}.\]
\end{itemize}

As a consistency check we use Goette's methods to compute the $s$ invariant introduced in \cite{Kreck1993} for the Aloff--Wallach spaces. Our computation shows that 
\[
	s(N_{k,l}) = \frac{kl(k+l)}{2^5\cdot 7},
\]
which is consistent with the result of Kreck and Stolz \cite{Kreck1993}.

In Section \ref{sect2}, we develop the necessary background on $G_2$-structures and Crowley--Nordstr\"om invariants to prove Theorem \ref{thm:xiintr}. In Section \ref{sect3}, we discuss the homogeneous $G_2$-structures and provide formulas for the $\bar{\nu}$ and $\xi$ invariants in the homogeneous case. In Section \ref{sect4}, we specialise the previous discussion to the Aloff--Wallach spaces and recall the diffeomorphism classification of these spaces from \cite{Kreck1991}. In Section \ref{sect5} and Appendix \ref{appx}, we carry out explicit computations of the terms constituting the $\bar{\nu}$ and $\xi$ invariants. In Section \ref{sect6}, we compare the nearly parallel $G_2$-structures given by 3-Sasakian structures with the associated squashed nearly-parallel proper $G_2$-structures. Finally, in Section \ref{sect7}, we gather some results about first Pontryagin class of nearly-parallel $G_2$-manifolds and note that first Pontryagin class of all known examples is a torsion class, which implies that $\xi$ is a complete invariant of homotopy classes of $G_2$-structures for these spaces.

\section{\texorpdfstring{Invariants of $G_2$-structures}{Invariants of G2-structures}}
\label{sect2}

\subsection{\texorpdfstring{The $\nu$ and $\bar{\nu}$ invariants of $G_2$-structures}{The nu invariants of G2-structures}}
\label{nuinv}

\begin{definition}
	A $G_2$-structure on a manifold $M^7$ is a choice of a 3-form $\varphi$, which is pointwise equivalent to the form 
\[
dx^{123} + dx^{145} + dx^{167} + dx^{246} - dx^{257} - dx^{347} - dx^{356}.
\]
Equivalently, the $G_2$-structure is determined by the choice of the orientation, metric $g$ and a unit spinor $\zeta \in \Gamma(SM)$.
\end{definition}

\begin{definition}
A nearly-parallel $G_2$-structure is a $G_2$-structure $\varphi$, satisfying 
\[
	d\varphi = \lambda {*}\varphi,
\]
for some $\lambda \neq 0$. 
\end{definition}

The metric induced by the nearly-parallel $G_2$-structure is Einstein with constant $7\cdot 24 \lambda^2$ \cite{Friedrich1997}.
 
\begin{definition}[{\cite[Definition 1.6]{Crowley2015}}]
    Let $(g,\zeta)$ be a $G_2$-structure on the closed manifold $M^7$. Let $W^8$ be its spin coboundary, i.e. $M = \partial W$ and the spin structure on $W$ restricts to the spin structure on $M$. Let $\bar{\zeta}$ be the section of $\Sigma^+W$ transversal to zero section such that its restriction to $M$ is $\zeta$. Then the $\nu$ invariant is defined as:
    \[
	    \nu = \chi(W) - 3\sigma(W) -2n(\bar{\zeta}) \mod 48,
    \]
    where $n(\zeta)$ is the number of zeroes of $\zeta$ counted with signs.
\end{definition}

This quantity is invariant under the deformations of the $G_2$-structure.

\begin{definition}[{\cite[Definition 1.6]{Crowley2025}}]
	Let $(g,\zeta)$ be a $G_2$-structure on the closed manifold $M^7$. Let $g^{SM}$ be the metric on the spinor bundle $SM$ and $\nabla^{SM}$ be the connection on $SM$ induced by the Levi-Civita connection on $TM$. The $\bar{\nu}$ invariant is defined as:
\begin{equation*}
	\bar{\nu}(\varphi) = 2\int_M \zeta^*\psi(\nabla^{SM},g^{SM}) - 24\eta(D_M) + 3\eta(B_M).
\end{equation*}

Here $\psi \in \mathcal{D}_7(SM)$ is the Mathai-Quillen current on the bundle $SM$ \cite{Bismut1992}, $D_M$ is the Dirac operator, and $B_M$ is the odd signature operator.

\end{definition}

This expression is invariant under the deformations of the $G_2$-structures preserving positive scalar curvature. Moreover, in \cite{Crowley2025}, it was proven that
\begin{equation*}
	\bar{\nu}(\varphi) + 24\dim \ker(D_M) \equiv \nu(\varphi) \mod 48.
\end{equation*}

\subsection{Mathai-Quillen current}
\label{MQcur}
We recall some of the properties of the Mathai-Quillen current that needed in the proof of Theorem \ref{thm:xiintr} and Section \ref{sect3}, for more detailed discussion one can use \cite{Bismut1992}.

\begin{lemma}[{\cite[Theorem 3.7]{Bismut1992}}]
Let $\nabla^1$, $\nabla^2$ be two metric connections on the vector bundle $SM$. The Mathai-Quillen current satisfies the following transgression formula:

\begin{equation}
	\label{eq:trans}
	\psi(\nabla^1,g) - \psi(\nabla^2,g) = \pi^*\widetilde{e}(\nabla^2,\nabla^1, g) \text{ \textnormal{modulo exact currents,}}
\end{equation}
where $\widetilde{e}(\nabla^2,\nabla^1, g) \in \Omega^7(M,\mathbb{R})$ is the second characteristic form associated to the Euler class. 
\end{lemma}

We will also need the following lemma 
\begin{lemma}[{\cite[Lemma 1.3]{Crowley2025}}]
	Let $\zeta \in \Gamma(SM)$. If $\zeta$ is parallel with respect to $\nabla$, then
	\[
		\zeta^*\psi(\nabla, g) = 0.
	\]
	\label{lem:red}
\end{lemma}

\subsection{\texorpdfstring{The $s$ invariant}{The s invariant}}
\label{sinv}
In this section we briefly discuss the $s$ invariant defined in the \cite{Kreck1988}. It was used to distinguish different connected components of the moduli space of the metrics with positive sectional (scalar) curvature. Our discussion of $s$ is restricted to dimension 7. 

The $s$ invariant in dimension $7$ is defined as:
\[
	s = \frac{\eta(B)}{2^5\cdot 7} + \frac{\eta(D)+h(D)}{2} - \frac{1}{2^7\cdot 7} \int_M p_1\wedge h.
\]

This expression reduces to the Eells-Kuiper invariant $\mu$ modulo $\mathbb{Z}$:

\[
	s \equiv \mu = \frac{p_W^2 - \sigma(W)}{2^5\cdot 7} \in \mathbb{Q}/\mathbb{Z},
\]
where $W$ is an $8$-dimensional coboundary of $M$, i.e. $M = \partial W$, and $p_W = \frac{p_1(W)}{2}$ is a spin characteristic class. 

The Eells-Kuiper invariant is a diffeomorphism invariant of the manifold, and so it doesn't depend on the choice of the coboundary. On the other hand, the $s$ invariant depends on the metric and can jump by an integer under deformations of the metric that leave connected component of metrics of positive scalar curvature. Note, that this invariant depends only on the metric and does not detect $G_2$-structures.

\subsection{\texorpdfstring{The $\xi$ invariant}{The xi invariant}}

In \cite{Crowley2015} the authors defined the $\xi$ invariant for the $G_2$-structures on arbitrary spin 7-manifold. In the case when the manifold $M$ has $p_1(M)$ rationally trivial this invariant can be defined in the following way. 
\begin{definition}
	Let $(M,\varphi)$ be a 7-manifold with $G_2$-structure, let $\zeta$ be the corresponding unit spinor. Let $W$ be its spin coboundary and let $\bar{\zeta}$ be the section of $\Sigma^+W$, such that its restriction to $M$ is $\zeta$. Additionally, assume that $p_1(M) = 0 \in H^4(M,\mathbb{Q})$. Then $\xi$ is defined as:
\begin{align*}
	\xi(\varphi) &= 7\left(\chi(W) - 3\sigma(W) - 2n(\bar{\zeta})\right) + 3\cdot \frac{p_W^2 - \sigma(W)}{2}\\
		     &= 7\chi(W) - \frac{45}{2}\sigma(W) - 14 n(\bar{\zeta}) + \frac{3}{2}p_W^2 \in \mathbb{Q},
\end{align*}
where $p_W = \frac{p_1(W)}{2}$ is the spin characteristic class.
\end{definition}

Similarly to the $\nu$ invariant in the case when $p_1 = 0 \in H^4(M,\mathbb{Q})$ the $\xi$ invariant admits an intrinsic formula.

\begin{proof}[Proof of Theorem \ref{thm:xiintr}]
	We rewrite the terms $\xi(W) - 3\sigma(W) - 2n(\bar{s})$ and $p_W^2 - \sigma(W)$ using the following relations and then combine them together to obtain an intrinsic formula for the $\xi$ invariant.
\begin{align}
	\label{ref:1}
	L &= \frac{1}{45}(7p_2 - p_1^2)\\
	\label{ref:2}
	\sigma(W) &= \int_W L - \eta(B_M)\\
	\label{ref:3}
	\widehat{A} &= \frac{1}{5760}(7p_1^2 - 4p_2)\\
	\label{ref:4}
	\ind(D^+_W) &= \int_W \widehat{A} - \frac{\eta(D)+h(D)}{2}\\
	\label{ref:5}
	n(\bar{\zeta}) &= \int_W e(\nabla^{S^+}) - \int_M \zeta^*\psi \\
	\label{ref:6}
	2 \int_W e(\nabla^{S^+}) &= 48 \int_W\widehat{A} + \int_W e(\nabla^{TW}) - 3\int_W L
\end{align}

Repeating the argument from \cite{Kreck1988}, let $h(M) \in \Omega^3(M)$ such that $dh = p_1(M)$ (It exists because $H^4(M,\mathbb{Q}) = 0$). Extend $h(M)$ to a 3-form on $W$ and denote this form by $h(W)$.
\begin{equation*}
	\begin{aligned}
		\langle p_1^2(W), [W,\partial W] \rangle &= \langle (p_1(W)-dh(W))\cup p_1(W), [W,\partial W] \rangle \\
							 &=\int_W (p_1(W) - dh(W)) \wedge p_1(W) \\
							 &= \int_W p_1^2 - \int dh(W) \wedge p_1(W)\\
							 &=\int_W p_1^2 - \int_M h(M) \wedge p_1(W).
	\end{aligned}
\end{equation*}

Hence, we have:
\begin{align}
	p_W^2 &= \frac{1}{4}\int_W p_1^2 - \frac{1}{4} \int_M p_1(M)\wedge h(M)
	\label{ref:8}
\end{align}

Now, we compute both parts constituting the $\xi$ invariant separately. 

Following the computation from \cite[Theorem 1.2]{Crowley2025} we obtain:

\begin{align*}
\chi(W) - 3\sigma(W) - 2n(\bar{\zeta}) &\overset{\eqref{ref:5}}{=} \chi(W) - 3\sigma(W) - 2\int_W e(\nabla^{S^+}) + 2\int_M \zeta^* \psi\\
					   &\overset{\eqref{ref:6}}{=} 3\int_W L -48\int_W \widehat{A} -3\sigma(W) + 2 \int_M \zeta^* \psi\\
					   &\overset{\eqref{ref:2}}{=} 2\int_M \zeta^*\psi + 3\eta(B_M) - 48 \int_W \widehat{A}
\end{align*}

Similarly, following the approach from \cite[Proposition 2.1]{Kreck1988}:
\begin{align*}
	p_W^2 - \sigma(W) &\overset{\eqref{ref:8}}{=} \frac{1}{4}\int_W p_1^2(W) - \frac{1}{4}\int_M p_1(\nabla^{LC})\wedge h(\nabla^{LC}) - \sigma(W) \\
			  &\overset{\eqref{ref:2}}{=} \frac{1}{4}\int_W p_1^2(W) - \frac{1}{4}\int_M p_1(\nabla^{LC})\wedge h(\nabla^{LC}) - \int_W L + \eta(B_M)\\
			  &\overset{\eqref{ref:1}}{=} \frac{1}{4}\int_W p_1^2(W) - \frac{1}{4}\int_M p_1(\nabla^{LC})\wedge h(\nabla^{LC}) + \frac{1}{45} \int_W p_1^2 -\frac{7}{45} \int p_2+ \eta(B_M)\\
			  &\overset{\eqref{ref:3}}{=} \frac{49}{4\cdot 45}\int_W p_1^2(W) - \frac{1}{4}\int_M p_1(\nabla^{LC})\wedge h(\nabla^{LC}) - \frac{7}{45} \left(\frac{7}{4} \int_W p_1^2(W) - 1440 \int_W \widehat{A}\right)+ \eta(B_M)\\
			  &= -\frac{1}{4}\int_M p_1(\nabla^{LC})\wedge h(\nabla^{LC}) + 7\cdot 32 \int_W \widehat{A} + \eta(B_M) 
\end{align*}

Combining the two formulas we have:

\begin{align*}
	\xi(\varphi) &= 7 \cdot \left(2\int_M \zeta^*\psi + 3\eta(B_M) - 48 \int_W \widehat{A}\right) + \frac{3}{2} \cdot \left(-\frac{1}{4}\int_M p_1(\nabla^{LC})\wedge h(\nabla^{LC}) + 7\cdot 32 \int_W \widehat{A}+ \eta(B_M)\right)\\
		     &= 7\cdot \left(2\int_M \zeta^*\psi + 3\eta(B_M)\right) + \frac{3}{2} \cdot \left(-\frac{1}{4}\int_M p_1(\nabla^{LC})\wedge h(\nabla^{LC})+ \eta(B_M)\right) - 7\cdot 48 \int_W \widehat{A} + 7 \cdot 48 \int_W \widehat{A}\\
		     &= 7\cdot \left(2\int_M \zeta^*\psi + 3\eta(B_M)\right) + \frac{3}{2} \cdot \left(-\frac{1}{4}\int_M p_1(\nabla^{LC})\wedge h(\nabla^{LC})+ \eta(B_M)\right)\\  
		     &= 14 \int_M \zeta^* \psi + \frac{45}{2} \eta(B_M) - \frac{3}{8} \int_M p_1(\nabla^{LC}) \wedge h(\nabla^{LC}).
\end{align*}
\end{proof}

This formula can be rewritten as: 
\[
	\xi = 7\bar{\nu} + 7h(D) + 12\cdot 28s,
\]
where $s$ is the invariant from the Section \ref{sinv}.

\begin{remark}
Note that both $\bar{\nu}$ and $s$ were only invariant up to deformations within the class of metrics with positive scalar curvature, because they included terms $\eta(D) + h(D)$, which can jump by an integer. However, if we considered $\bar{\nu} \mod 48$ or $s \mod \mathbb{Z}$, we would get genuine invariants of $G_2$-structure and smooth structure, respectively. By contrast, the formula for the $\xi$ invariant that we derived is already a genuine invariant of $G_2$-structures, since the terms $\eta(D) + h(D)$ from $\bar{\nu}$ and $s$ cancel out.
\end{remark}
\begin{remark}
	\label{rm:neginv}
    For all of these invariants we have:
    \begin{align}
          \nu(-\varphi) &= -\nu(\varphi),\\
	  \bar{\nu}(-\varphi) &= -\bar{\nu}(\varphi),\\
	  \xi(-\varphi) &= -\xi(\varphi).
    \end{align}
\end{remark}

\subsection{\texorpdfstring{Homotopy classes of $G_2$-structures}{Homotopy classes of G2-structures}}
\label{sect:homcl}
Let $\mathcal{G}_2$ denote the space of $G_2$-structures compatible with the chosen spin structure on $M$. By \cite[Lemma 1.1]{Crowley2025}, the set of connected components of $\mathcal{G}_2$ $\pi_0\mathcal{G}_2$ can be identified with $\mathbb{Z}$. The group of spin-diffeomorphisms $\Diff(M)$ acts on $\mathcal{G}_2$ by pull-back, and we denote the quotient $\bar{\mathcal{G}}_2 = \mathcal{G}_2/\Diff(M)$. The set $\pi_0 \bar{\mathcal{G}}_2 = \pi_0\mathcal{G}_2/ \pi_0 \Diff(M)$ is called the set of deformation classes of $G_2$-structures on $M$.

When $p_1(M)$ is rationally trivial the diffeomorphism group acts trivially on $\mathcal{G}_2$ and one has the following theorem:

\begin{theorem}[{\cite[Theorem 1.11]{Crowley2025}}]
	If $p_1(M) = 0 \in H^4(M,\mathbb{Q})$, then $\pi_0 \mathcal{G}_2 = \pi_0 \bar{\mathcal{G}}_2 = \mathbb{Z}$.
\end{theorem}

In particular, in this case $\xi$ is a complete homotopy invariant of $G_2$-structures (cf. \cite[Section 1.6]{Crowley2025}). By contrast, the $\nu$ invariant alone is not enough to determine whether $G_2$-structures are deformation equivalent as will be seen in the case of Aloff--Wallach spaces.

\section{\texorpdfstring{Homogeneous $G_2$-structures}{Homogeneous G2-structures}}
\label{sect3}
In this section we derive formulas for the $\bar{\nu}$, $s$, and $\xi$ invariants of homogeneous $G_2$-structures in terms of representation-theoretic data. The main tool is the Goette's formulas for the $\eta$ invariants of homogeneous spaces.

Throughout the paper we assume that $G$ and $H$ are compact Lie groups.
\subsection{Reductive connection}
\label{nabla0}
Let $G/H$ be a homogeneous space. Any vector bundle $E$ over $G/H$ is of the form $E = G \times_\kappa V$ for some $H$-representation $\kappa : H \to \Aut(V)$ \cite{Goette1999}. Any homogeneous section $\zeta$ of such a bundle can be identified with $H$-equivariant map $\hat{\zeta}: G \to V$ via

\[
	\zeta([g]) = [g,\hat{\zeta}(g)].
\]

The reductive connection is defined as \cite{Goette2009}:
\[
	\widehat{\nabla^0_V\zeta}(g) = \widehat{V}\hat{\zeta}(g) = \frac{d}{dt}\bigg|_{t=0}\hat{\zeta}(ge^{t\widehat{V}}).
\]

\begin{lemma}
	Any homogeneous section $\zeta \in \Gamma(E)$ is parallel with respect to $\nabla^0$.
	\label{lem:redpar}
\end{lemma}
\begin{proof}%[Proof of Lemma \ref{lem:redpar}]
	Section $\zeta$ is homogeneous if and only if it is $G$-invariant, that is,

\[
	l_g\zeta = \zeta \ \forall g \in G.
\] 

Equivalently,
\[
	\hat{\zeta}(g_0) = \widehat{l_g\zeta}(g_0) = \hat{\zeta}(g^{-1}g_0) \ \forall g \in G.
\]
	Hence, $\hat{\zeta}$ is constant on $G$, therefore 
\[
	\nabla^0\zeta = 0.
\]
\end{proof}

\begin{remark}
	\label{rm:homsections}
	The space of sections parallel with respect to $\nabla^0$ can be identified with the subspace $V^H \subset V$ consisting of vectors fixed by $H$ action.
\end{remark}

\begin{lemma}
All of the homogeneous $G_2$-structures on $G/H$ inducing the same orientation are homotopic.
\label{rmk:homstr}
\end{lemma}
\begin{proof}%[Proof of Lemma \ref{rmk:homstr}]
	Let $(S,\widetilde{\pi})$ and $(T,\kappa)$ denote the $H$-representations associated to the spinor bundle $SM$ and tangent bundle $TM$, respectively.

	A homogeneous $G_2$-structure is determined by a homogeneous Riemannian metric together with a homogeneous unit spinor. By Remark \ref{rm:homsections}, homogeneous sections correspond precisely to $H$-invariant vectors. Consequently, homogeneous $G_2$-structures inducing a fixed orientation are in one-to-one correspondence with pairs
\[
(g,\zeta)\in \mathcal{P}(T)^H\times \mathbb{S}(S^H),
\]
where $\mathcal{P}(T^H)\subset (\odot^2(T^*))^H$ denotes the cone of positive-definite symmetric bilinear forms and $\mathbb{S}(S^H)$ is the unit sphere in $S^H$.

Since $\mathcal{P}(T)^H$ is convex, any two homogeneous metrics can be joined by a path of homogeneous metrics. If $\dim S^H > 1$, then the sphere $\mathbb{S}(S^H)$ is connected, and therefore any two homogeneous unit spinors can also be joined by a path. Combining these paths yields a homotopy between the corresponding $G_2$-structures.

It remains to consider the case $\dim S^H = 1$. In this case $\mathbb{S}(S^H)=\{\pm \zeta\}$. Following \cite{Crowley2025}, the spinor $\zeta$ determines a splitting
\[
SM \cong \underline{\mathbb{R}}\oplus TM.
\]
Since the Euler class of an oriented $7$-manifold vanishes, the spinor bundle contains a trivial oriented $2$-plane bundle $K\subset SM$ containing $\zeta$. Rotating within $K$ gives a path from $\zeta$ to $-\zeta$, and hence a homotopy between the corresponding $G_2$-structures. In general, this homotopy need not remain within the space of homogeneous $G_2$-structures.
\end{proof}
\subsection{\texorpdfstring{The $\eta$ invariants of homogeneous spaces}{The eta invariants of homogeneous spaces}}
\label{etainv}
To compute the $\eta$ invariants, we use the following results:
\begin{theorem}[{\cite[Theorems 2.33, 2.34]{Goette2009}}]
	Let $G/H$ be a homogeneous space with the normal metric $g$. Then, the following formulas for the Dirac operator $D$ and the odd signature operator $B$ hold:
\begin{equation}
	\label{eq:etaD}
	\eta(D) = I_D + 2\int_M \widetilde{\widehat{\mathrm{A}}}(\nabla^0,\nabla^{TM}) + J_D.
\end{equation}

\begin{equation}
	\label{eq:etaB}
	\eta(B) = I_B + \int_M \widetilde{L}(\nabla^0,\nabla^{TM}) + J_B.
\end{equation}

Here, $\widetilde{\widehat{\mathrm{A}}}$ and $\widetilde{L}$ are the secondary characteristic forms of the $\widehat{\mathrm{A}}$-genus and $L$-genus. 

The terms $I$ and $J$ depend purely on the representation-theoretic data of $G/H$ and are explained in section \ref{sect5}.
\end{theorem}

It turns out that these formulas are well-suited for the computation of $\bar{\nu}$ invariants, which we discuss in the next section.

\subsection{\texorpdfstring{$\nu$ invariant in the homogeneous case }{nu invariant in the homogeneous case}}
\label{nuhom}

We derive the formula for the $\bar{\nu}$ invariant of the homogeneous $G_2$-structure inducing the \textit{normal metric} in terms of the reductive connection $\nabla^0$. This follows the approach of [\cite{Crowley2025}, section 1.3].

\begin{prop}
	Let $G/H$ be a homogeneous space and let $\varphi$ be a homogeneous $G_2$-structure inducing the normal metric. Then the $\bar{\nu}$ invariant can be computed as:
	\begin{equation}
		\bar{\nu}(\varphi) = -24I_D + 3I_B -24J_D + 3J_B.
	\end{equation}
	\label{prop:homnu}
\end{prop}
\begin{proof}%[Proof of Proposition \ref{prop:homnu}]
First we use formula \eqref{eq:trans} to rewrite the Mathai-Quillen term as:

\begin{equation*}
	2\int_M \zeta^*\psi(\nabla^{SM},g^{SM}) = 2\int_M \zeta^*\psi(\nabla^0,g^{SM}) + 2\int_M \widetilde{e}(\nabla^0,\nabla^{SM}).
\end{equation*}

Using the standard formulas for the Euler class, this becomes 
\begin{equation*}
	= 2\int_M \zeta^*\psi(\nabla^0,g^{SM}) + 48 \int_M \widetilde{\widehat{\mathrm{A}}}(\nabla^0,\nabla^{SM}) - 3 \int_M \widetilde{L}(\nabla^0,\nabla^{SM}).
\end{equation*}

Since $\zeta$ is homogeneous, it is parallel with respect to $\nabla^0$. By Lemma \ref{lem:red} we have
\[
	\zeta^*\psi(\nabla^0,g^{SM}) = 0.
\]

Consequently,
\begin{equation*}
	2\int_M \zeta^*\psi(\nabla^{SM},g^{SM}) = 48 \int_M \widetilde{\widehat{\mathrm{A}}}(\nabla^0,\nabla^{SM}) - 3 \int_M \widetilde{L}(\nabla^0,\nabla^{SM}).
\end{equation*}

We now apply the formulas \eqref{eq:etaD} and \eqref{eq:etaB} for the $\eta$ invariants of the Dirac and odd-signature operators:

\begin{align*}
	\bar{\nu}(\varphi) &= 48 \int_M \widetilde{\widehat{\mathrm{A}}}(\nabla^0,\nabla^{SM}) - 3 \int_M \widetilde{L}(\nabla^0,\nabla^{SM}) -\\
			   &-24I_D - 48\int_M \widetilde{\widehat{\mathrm{A}}}(\nabla^0,\nabla^{SM}) - 24J_D +\\
	&+ 3I_B + 3\int_M \widetilde{L}(\nabla^0,\nabla^{SM}) + 3J_B =\\ 
	&=-24I_D + 3I_B -24J_D + 3J_B.\qedhere
\end{align*}

\end{proof}

This formula was used in \cite[Section 1.3]{Crowley2025} to compute the $\bar{\nu}$ invariant of the homogeneous $G_2$ structure on the Berger space $SO(5)/SO(3)$. It shows that, in the homogeneous case, $\bar{\nu}$ can be expressed purely in terms of representation-theoretic data of the pair $(G,H)$.

We will compute these terms for the Aloff--Wallach spaces in section \ref{sect5}.
\subsection{\texorpdfstring{The $s$ invariant in the homogeneous case}{The s invariant in the homogeneous case}}
In this section we state the formula for the $s$ invariant in the homogeneous case.

\begin{prop}
	\label{prop:homs}
    Let $G/H$ be a homogeneous 7-manifold with normal metric $g$. Additionally, assume that $G/H$ is non-flat. Then $s$ invariant can be computed as:
    \begin{equation*}
        s(g) = \frac{I_B + J_B}{2^5\cdot 7} + \frac{I_D + J_D}{2} - \frac{1}{2^7\cdot 7}\int_M p_1(\nabla^0)\wedge h(\nabla^0).
    \end{equation*}
\end{prop}
\begin{proof}%[Proof of Proposition \ref{prop:homs}]
    We recall the approach from \cite{Goette2004} used to compute the Eells-Kuiper invariant in the homogeneous case. Recall the definition of the $s$ invariant:

\[
s = \frac{\eta(B)}{2^5\cdot 7} +\frac{\eta(D)+h(D)}{2} - \frac{1}{2^7\cdot 7}\int_M p_1(\nabla^{TM})\wedge h(\nabla^{TM}).
\]

Since $G$ is assumed to be compact and metric $g$ is non-flat, then scalar curvature is positive \cite{Besse1987} and $h(D) = 0$.

From \cite{Goette2004} we can also rewrite the term 
\begin{align*}
	\frac{1}{2^7\cdot 7}\int_M p_1(\nabla^{TM})\wedge h(\nabla^{TM}) &= \frac{1}{2^7\cdot 7}\int_M p_1(\nabla^0)\wedge h(\nabla^0) \\
									&+\int_M  \widetilde{\widehat{\mathrm{A}}}(\nabla^0,\nabla^{TM}) + \frac{1}{2^5\cdot 7}\int_M \widetilde{L}(\nabla^0,\nabla^{TM}).
\end{align*}

Then 

\begin{align*}
	s &= \frac{\eta(B)}{2^5\cdot 7} + \frac{\eta(D)+h(D)}{2} - \frac{1}{2^7\cdot 7} \int_M p_1\wedge h\\
	  &= \frac{I_B + J_B}{2^5\cdot 7} + \frac{\int_M \widetilde{L}(\nabla^0,\nabla^{TM})}{2^5\cdot 7} + \frac{I_D + J_D}{2} + \frac{2\int_M \widetilde{\widehat{\mathrm{A}}}(\nabla^0,\nabla^{TM})}{2}\\
	  &- \frac{1}{2^7\cdot 7}\int_M p_1(\nabla^0)\wedge h(\nabla^0)  - \int_M  \widetilde{\widehat{\mathrm{A}}}(\nabla^0,\nabla^{TM}) - \frac{1}{2^5\cdot 7}\int_M \widetilde{L}(\nabla^0,\nabla^{TM})\\
	  &= \frac{I_B + J_B}{2^5\cdot 7} + \frac{I_D + J_D}{2} - \frac{1}{2^7\cdot 7}\int_M p_1(\nabla^0)\wedge h(\nabla^0). \qedhere
\end{align*}
\end{proof}

\begin{remark}
	This formula was used in \cite{Goette2004} to compute the Eells-Kuiper invariant of the Berger space. 
\end{remark}

\subsection{\texorpdfstring{The $\xi$ invariant in the homogeneous case}{The xi invariant in the homogeneous case}}

In this section we combine the methods from two previous sections to derive a formula for the $\xi$ invariant in the homogeneous case. 
\begin{prop}
	\label{prop:homxi}
    Let $G/H$ be a homogeneous space, let $\varphi$ be a homogeneous $G_2$-structure inducing the normal metric, then its $\xi$ invariant can be computed as:
    \begin{equation*}
        \xi(\varphi) = \frac{45}{2}\left(I_B + J_B\right) - \frac{3}{8}\int_M p_1(\nabla^0)\wedge h(\nabla^0).
    \end{equation*}
\end{prop}
\begin{proof}%[Proof of Proposition \ref{prop:homxi}]
Recall that,

\[
	\xi = 14 \int_M \zeta^* \psi + \frac{45}{2} \eta(B_M) - \frac{3}{8} \int_M p_1(\nabla^{LC}) \wedge h(\nabla^{LC}).
\]

Analogously to the Proposition \ref{prop:homs}, we can rewrite this as:
\begin{align*}
	\xi &= 7 \cdot \left(48 \int_M \widetilde{\widehat{\mathrm{A}}}(\nabla^0,\nabla^{SM}) - 3 \int_M \widetilde{L}(\nabla^0,\nabla^{SM})\right) \\
	    &+ \frac{45}{2}\left(I_B + \int_M \widetilde{L}(\nabla^0,\nabla^{TM}) + J_B\right)\\
	    &- \frac{3}{8}\left(\int_M p_1(\nabla^0)\wedge h(\nabla^0) +2^7 \cdot 7\int_M  \widetilde{\widehat{\mathrm{A}}}(\nabla^0,\nabla^{TM}) + 4\int_M \widetilde{L}(\nabla^0,\nabla^{TM})\right)\\
	    &= \frac{45}{2}\left(I_B + J_B\right) - \frac{3}{8}\int_M p_1(\nabla^0)\wedge h(\nabla^0). \qedhere
\end{align*}

\end{proof}
\section{\texorpdfstring{Aloff--Wallach spaces $N_{k,l}$}{Aloff--Wallach spaces N(k,l)}}
\label{sect4}
\subsection{Geometry of Aloff--Wallach spaces}
\label{AWspaces}

Let $(k,l)$ be the pair of integer numbers such that $k \neq \pm l$, $l \neq \pm(k+l)$, $k+l \neq \pm k$, and $k$ and $l$ are coprime. 

The Aloff--Wallach spaces $N_{k,l}$ are defined as quotients $SU(3)/S^1_{k,l}$, where the subgroup $S^1_{k,l}$ is given as $\left\{\diag\left(e^{ikx},e^{ilx},e^{-i(k+l)x}\right)\right\}$
We will also assume that $k,l > 0$ (other cases can be obtained from this one by the change of orientation and permutations of $(k,l,-k-l)$)

\begin{remark}
	While we restrict our attention to the general case $k \neq \pm l$, $l \neq \pm(k+l)$, $k+l \neq \pm k$, the exceptional Aloff--Wallach spaces also admit homogeneous $G_2$-structures, but the space of such structures is larger than in the general case. Nevertheless, our methods still can be applied in these cases. We consider them in Section \ref{sect:except}.
\end{remark}

First, we describe the structure of the Aloff--Wallach spaces.
Fix the metric on $\mathfrak{su}(3)$ by $\langle X,Y \rangle = -\tr(XY)$.
Let $\mathfrak{su}(3) = \mathfrak{m} \oplus \mathfrak{u}(1)_{k,l}$ be the orthogonal decomposition. Choose the following basis for the subspace $\mathfrak{m}:$
\begin{align}
e_1 &= \frac{1}{\sqrt{2}}
\begin{pmatrix}
0 & 1 & 0 \\
-1 & 0 & 0 \\
0 & 0 & 0
\end{pmatrix},
&
e_5 &= \frac{i}{\sqrt{2}}
\begin{pmatrix}
0 & 1 & 0 \\
1 & 0 & 0 \\
0 & 0 & 0
\end{pmatrix}, \notag
\\[1em]
e_2 &= \frac{1}{\sqrt{2}}
\begin{pmatrix}
0 & 0 & 0 \\
0 & 0 & 1 \\
0 & -1 & 0
\end{pmatrix},
&
e_6 &= \frac{i}{\sqrt{2}}
\begin{pmatrix}
0 & 0 & 0 \\
0 & 0 & 1 \\
0 & 1 & 0
\end{pmatrix},
\label{eq:basis}
\\[1em]
e_3 &= \frac{1}{\sqrt{2}}
\begin{pmatrix}
0 & 0 & -1 \\
0 & 0 & 0 \\
1 & 0 & 0
\end{pmatrix},
&
e_7 &= \frac{i}{\sqrt{2}}
\begin{pmatrix}
0 & 0 & 1 \\
0 & 0 & 0 \\
1 & 0 & 0
\end{pmatrix},\notag
\\[1em]
e_4 &= \frac{i}{\sqrt{6}\,\sqrt{k^2 + l^2 + kl}}
\begin{pmatrix}
2l + k & 0 & 0 \\
0 & -2k - l & 0 \\
0 & 0 & k - l
\end{pmatrix}.\notag
\end{align}
The subspace $\mathfrak{m}^\perp = \mathfrak{u}(1)_{k,l}$ is generated by
\[
e_8 = \frac{i}{\sqrt{2}\,\sqrt{k^2 + l^2 + kl}}
\begin{pmatrix}
k & 0 & 0 \\
0 & l & 0 \\
0 & 0 & -k - l
\end{pmatrix}.
\]
Denote the adjoint action of $S^1$ as $\pi$. Under the action of $\pi$ vector $e_4$ is fixed, while the planes $\langle e_1,e_5 \rangle$, $\langle e_2,e_6 \rangle$, and $\langle e_3,e_7 \rangle$ carry weights $i(k-l)$, $i(2l+k)$, and $i(-2k-l)$ respectively.

The tangent bundle is given as $TN_{k,l} = SU(3)\times_\pi \mathfrak{m}$.

Let $r = \sqrt{6(k^2+kl+l^2)}$. The multiplication table for the commutators is given as:

\begin{paracol}{2}
\begin{align*}
[e_1,e_2]_\mathfrak{m}&=-\tfrac{1}{\sqrt2}e_3,\\
[e_1,e_3]_\mathfrak{m}&=\tfrac{1}{\sqrt2}e_2,\\
[e_1,e_4]_\mathfrak{m}&=-\tfrac{3(k+l)}{r}e_5,\\
[e_1,e_5]_\mathfrak{m}&=\tfrac{3(k+l)}{r}e_4,\\
[e_1,e_6]_\mathfrak{m}&=\tfrac{1}{\sqrt2}e_7,\\
[e_1,e_7]_\mathfrak{m}&=-\tfrac{1}{\sqrt2}e_6.\\
\\
[e_3,e_4]_\mathfrak{m}&=\tfrac{3l}{r}e_7,\\
[e_3,e_5]_\mathfrak{m}&=\tfrac{1}{\sqrt2}e_6,\\
[e_3,e_6]_\mathfrak{m}&=-\tfrac{1}{\sqrt2}e_5,\\
[e_3,e_7]_\mathfrak{m}&=-\tfrac{3l}{r}e_4.\\
\\
[e_6,e_4]_\mathfrak{m}&=-\tfrac{3k}{r}e_2,\\
[e_6,e_7]_\mathfrak{m}&=\tfrac{1}{\sqrt2}e_1.
\end{align*}

\switchcolumn

\vspace{0.6ex}
\begin{align*}
[e_2,e_3]_\mathfrak{m}&=-\tfrac{1}{\sqrt2}e_1,\\
[e_2,e_4]_\mathfrak{m}&=\tfrac{3k}{r}e_6,\\
[e_2,e_5]_\mathfrak{m}&=-\tfrac{1}{\sqrt2}e_7,\\
[e_2,e_6]_\mathfrak{m}&=\tfrac{-3k}{r}e_4,\\
[e_2,e_7]_\mathfrak{m}&=\tfrac{1}{\sqrt2}e_5.\\
\\
\\
[e_5,e_4]_\mathfrak{m}&=\tfrac{3(k+l)}{r}e_1,\\
[e_5,e_6]_\mathfrak{m}&=\tfrac{1}{\sqrt2}e_3,\\
[e_5,e_7]_\mathfrak{m}&=-\tfrac{1}{\sqrt2}e_2.
\\
\\
\\
[e_7,e_4]_\mathfrak{m}&= -\frac{3l}{r}e_3.
\end{align*}

\end{paracol}

Let $\widetilde{\pi}: \mathfrak{h} \to \End(S)$ denote the spin representation induced from the isotropy representation $\pi$. Then the spinor bundle is given as $SM = G \times_{\widetilde{\pi}} S$. The homogeneous bundle $\Omega^{ev}M$ in dimension seven is isomorphic to the bundle $G \times_{\widetilde{\pi} \otimes \hat{\widetilde{\pi}}} S\otimes S$, where $\hat{\widetilde{\pi}}$ denotes the representation isomorphic to $\widetilde{\pi}$ acting on the second factor.

It is easy to check that in the case of the Aloff--Wallach space $N_{k,l}$ the weights of $\widetilde{\pi}$ and $\hat{\widetilde{\pi}}$ are $(0,0, \pm i(k-l), \pm i(2k+l), \pm i(2l+k))$.

\subsection{Diffeomorphism classification of Aloff--Wallach spaces}
\label{sect:classif}
In this section we recall the classification of Aloff--Wallach spaces proven in \cite{Kreck1991}.

\begin{theorem}[{\cite{Kreck1991}}]
	\label{ref:classif}
    Let, $k,l$ and $k',l'$ be pairs of coprime integers, then there exist orientation-preserving diffeomorphism between spaces $N_{k,l}$ and $N_{k',l'}$ if and only if $k^2+kl+l^2 = k'^2 +k'l' + l'^2$ and $kl(k+l) \equiv k'l'(k'+l') \mod 2^5 \cdot 7^{\lambda(N)} \cdot 3 \cdot N$, where $N = (k^2+kl+l^2)$ and 
    \begin{equation*}
        \lambda(N) = \begin{cases}
		0 \text{ if } N \equiv 0 \mod 7,\\
		1 \text{ otherwise.}
        \end{cases}
    \end{equation*}
\end{theorem}

Non-trivial pairs of diffeomorphic Aloff--Wallach spaces were found in \cite{Kreck1991} and  \cite{Kreck1993}. Two such pairs are $N_{8287903, 14825547} \cong N_{-18401045, 21746838}$ and $N_{-4638861,582656} \cong N_{-2594149,5052965}$.

\begin{remark}
	The proof of this theorem in \cite{Kreck1991} among other things requires computing the Eells--Kuiper invariant, which is done via constructing explicit coboundaries. These coboundaries are not spin, which makes them unsuitable for computing $\nu$ and $\xi$ invariants in this case.
\end{remark}

\begin{remark}
	Note that in general permutations of the triple $(k,l,-k-l)$ produce diffeomorphic manifolds, but for general $k,l$ these spaces are not equal, only diffeomorphic to each other. However, in the case $(k,l) = (1,-1)$, spaces $N_{1,-1}$ and $N_{-1,1}$ are the same, except for orientation. 

	It follows that $N_{1,-1}$ admits a self-diffeomorphism reversing orientation. From the theorem \ref{ref:classif}, it actually follows that Aloff--Wallach space admits such diffeomorphism if and only if $kl(k+l) = 0$, which is precisely the case of $N_{1,-1}$.
\end{remark}

\subsection{\texorpdfstring{Homogeneous $G_2$-structures on Aloff--Wallach spaces}{Homogeneous G2-structures on Aloff--Wallach spaces}}
\label{AWg2}

Under our assumptions on $k,l$ the most general homogeneous metric is given by choosing the orthonormal basis of the form
\[
	(ae_1, ae_5, be_2, be_6, ce_3, ce_7, de_4).
\]

The associated $G_2$ structure is then chosen by identifying $\mathfrak{m}$ with $\Img\mathbb{O}$ as follows: $ae_1$ is identified with $i$, $ae_5$ with $ie$, $be_2$ with $j$, $be_6$ with $je$, $ce_3$ with $c'k + s'ke$, $ce_7$ with $-s'k + c'ke$ ($c' = \cos x, s' = \sin x$ for some $x$), $de_4$ with $e$.  According to \cite{Cabrera1996} every homogeneous $G_2$-structure on $N_{k,l}$ arises in this way. 

The $G_2$ 3-form is given as:
\begin{align*}
\varphi &= abc \, c' e_1 \wedge e_2 \wedge e_3 - abc \, s' e_1 \wedge e_2 \wedge e_7 + (a^2 d) e_1 \wedge e_4 \wedge e_5 \\
        &\quad - abc \, c' e_1 \wedge e_6 \wedge e_7 + abc \, s' e_1 \wedge e_3 \wedge e_6 + (b^2 d) e_2 \wedge e_4 \wedge e_6 \\
        &\quad + abc \, c' e_2 \wedge e_5 \wedge e_7 - abc \, s' e_2 \wedge e_3 \wedge e_5 + (c^2 d) e_3 \wedge e_4 \wedge e_7 \\
        &\quad - abc \, c' e_3 \wedge e_5 \wedge e_6 + abc \, s' e_5 \wedge e_6 \wedge e_7.
\end{align*}

The coclosed homogeneous $G_2$-structures are given precisely by the condition $s' = 0$. The space of such structures admits an obvious $\mathbb{Z}_2\times \mathbb{Z}_2$ symmetry. As shown in \cite{Ball2019} the space of such $G_2$-structures can therefore be identified with
\[
	(\mathbb{R}^+)^2 \times (\mathbb{R}\backslash \{0\})
\]
by fixing the signs of $a,b$ to be positive.

Up to scaling, there are four different nearly-parallel $G_2$-structures, which we denote $\pm \varphi_1, \pm \varphi_2$. We assume that $\varphi_1$ and $\varphi_2$ induce the chosen orientation on $N_{k,l}$. Since these are homogeneous $G_2$-structures they are homotopic, so 
\[
	\nu(\varphi_1) = \nu(\varphi_2)\ \text{and } \xi(\varphi_1) = \xi(\varphi_2).
\]
But, in fact, one can prove an even stronger statement:
\begin{lemma}
	The homogeneous nearly-parallel $G_2$-structures $\varphi_1$ and $\varphi_2$ are homotopic through a path of homogeneous $G_2$-structures inducing the metrics of positive scalar curvature. In particular, we have
	\[
		\bar{\nu}(\varphi_1) = \bar{\nu}(\varphi_2).
	\]
	\label{lem:nu}
\end{lemma}
The proof of Lemma \ref{lem:nu} is given in the next section.
\subsection{Scalar curvature}
\label{scalar}

In this section we will discuss the scalar curvature of homogeneous metrics on $N_{k,l}$.

Let $\{f_i\}$ be the dual basis to $\{e_i\}$. Recall that the most general homogeneous metric on $N_{k,l}$ is given by
\[
	g = a^2\left({f_1}^2 + f_5^2\right) + b^2\left(f_2^2 + f_6^2\right) + c^2\left(f_3^2 + f_7^2\right) + d^2f_4^2.
\]

We are interested in the sign of the scalar curvature, which is preserved under rescaling, so we consider the rescaled metric
\[
	g/d^2 = \lambda_1\left({f_1}^2 + f_5^2\right) + \lambda_2\left(f_2^2 + f_6^2\right) + \lambda_3\left(f_3^2 + f_7^2\right) + f_4^2.
\]
The scalar curvature of such a metric can by computed using the following result:
\begin{theorem}[{\cite[Theorem 3.4]{Park2013}}]

\begin{equation*}
S_{(\lambda_1,\lambda_2,\lambda_3)} =
\frac{-(\lambda_1^2 + \lambda_2^2 + \lambda_3^2) + 6(\lambda_1\lambda_2 + \lambda_2\lambda_3 + \lambda_3\lambda_1)}{6\,\lambda_1\lambda_2\lambda_3} - 
\frac{1}{8q}\left(\frac{(k+l)^2}{\lambda_1^2} + \frac{l^2}{\lambda_2^2} + \frac{k^2}{\lambda_3^2}\right),
\end{equation*}
where $q = k^2 + kl + l^2$. 

\end{theorem}

We write this as 
\[
	S(\lambda) = f(\lambda) - g(\lambda),
\]
where $f$ is homogeneous of degree $-1$ and $g$ is homogeneous of degree $-2$. Consequently, for any $t > 0$

\[
	S(t\lambda) = \frac{1}{t}f(\lambda) - \frac{1}{t^2}g(\lambda).
\]
\begin{lemma} 
	The space of homogeneous metrics with positive scalar curvature is connected. 
	\label{lem:connctd}
\end{lemma}

\begin{proof}%[Proof of Lemma \ref{lem:connctd}]
	Note that if for some $\lambda \in \mathbb{R}_+^3$ $f(\lambda) > 0$ then $S(t\lambda) > 0$ for all sufficiently large $t > 0$. Hence, the ray $\{t\lambda, t > 0\}$ intersects the set $\{S > 0\}$ in an unbounded interval. In particular, the set $\{S > 0\}$ is a cone over $\{S = 0\}$. Now, we project this set to the plane $\lambda_3 = 1$. The projection will be the same as projection of $\{f > 0\}$, which is connected. Since $S = 0$ is a smooth surface in $\mathbb{R}^3_+$ and its projection to the plane $\lambda_3 = 1$ is connected and one to one, the set \{S = 0\} is also connected. Since $\{S > 0\}$ is a cone over \{S = 0\}, it is also connected.
\end{proof}

\begin{proof}[Proof of Lemma \ref{lem:nu}]
	The argument is similar to the proof of the lemma \ref{lem:connctd}.

	Since both $\varphi_1$ and $\varphi_2$ are nearly-parallel, they induce metrics of positive scalar curvature. Additionally, $\widetilde{\pi}$ has at least two zero weights, implying that $\Sigma N_{k,l}$ has at least a 2-dimensional subspace of homogeneous sections. So, in particular, the space of homogeneous unit sections of the spinor bundle is connected. Consequently, we can find a path of $G_2$-structures connecting $\varphi_1$ and $\varphi_2$ through $G_2$-structures inducing the positive scalar curvature. 

	We also note that this argument shows that both $\varphi_1$ and $\varphi_2$ are homotopic to some homogeneous $G_2$-structure inducing the normal metric on $N_{k,l}$ (i.e. $a = b = c = d = 1$), since such a metric also has positive sectional curvature. 

	The corresponding formulas for the $\bar{\nu}$ and $\xi$ invariants now follow by their invariance under the deformations of $G_2$-structures inducing the metrics of positive scalar curvature. 
\end{proof}

\begin{remark}From the proof of Lemma \ref{lem:nu} we can see that to compute $\bar{\nu}$ invariants of $\varphi_i$, $i=1,2$ it is enough to compute them for some homogeneous $G_2$-structure inducing the normal metric. 
\end{remark}
%\begin{remark}
%	From the proof of Lemmas \ref{lem:nu} and \ref{lem:connctd} we can see that structures $-\varphi_-$ and $\varphi_+$ are homotopic through the path of $G_2$-structures inducing metrics with positive scalar curvatures. 
%\end{remark}

\section{Computations of the invariants}
\label{sect5}

In this section we compute the $I$ and $J$ terms appearing in the Goette's formulas for $\eta$ invariants in the case of the Aloff--Wallach spaces.

First we state the general formulas.

\begin{itemize}
	\item
Let $W_G$ denote the Weyl group of $G$, $\Delta^+_G$ the set of positive roots of $G$, $\mathfrak{t}, \mathfrak{s}$ be the maximal Cartan subalgebras inside $\mathfrak{g},\mathfrak{h}$. Let $\rho_G$ and $\rho_H$ be the half-sums of positive roots of $G$ and $H$. Let $\widehat{A}(z) = \tfrac{z/2}{\sinh(z/2)}$.

Take $E\in \mathfrak{s}^\perp \subset \mathfrak{h}$ be the positive unit vector, let $\delta \in -i\mathfrak{t}^*$ be the unique weight such that $-i\delta(E) > 0$ and $\delta(X) \in 2\pi i\mathbb{Z} \Leftrightarrow e^X \in S.$

Then, by \cite[Theorem 2.33]{Goette2009} the first term for the $\eta$ invariant of the Dirac operator is given as:
\begin{equation*}
\begin{aligned}
	I_D = 2 \sum_{w \in W_G} \frac{\sign(w)}{\delta(wX)} \biggl( &\prod_{\beta \in \Delta_G^+}\widehat{A}(\beta(wX)) e^{- \frac{\delta}{2}(wX)} -\\ 
	&- \prod_{\beta \in \Delta_G^+}\widehat{A}(\beta(wX|_\mathfrak{s}))e^{-\rho_H(wX|_{\mathfrak{s}})}\biggr) \cdot \prod\limits_{\beta \in \Delta^+_G}\frac{-1}{\beta(X)}\bigg|_{X = 0}.
\end{aligned}
\end{equation*}

Let $\hat{\widetilde{\pi}}: H \to \End(S)$ denote the action of $H$ inducing the $\Lambda^{ev}$ bundle on $G/H$. Let $\{\kappa_j\}$ be the weights of $\hat{\widetilde{\pi}}$. Take $\{\alpha_j\}$ to be unique weights in $i\mathfrak{t}^*$ such that ${\alpha_j}|_\mathfrak{s} = \kappa_j + \rho_H$ and $-i(\alpha_j-\delta)(E)< 0 \leqslant -i\alpha_j(E)$. We will call weight $\alpha_j$ a \textit{lift} of the weight $\kappa_j$.

Then, by \cite[Theorem 2.34]{Goette2009} the first term for the $\eta$ invariant of the odd-signature operator is given as:
\begin{equation*}
\begin{aligned}
	I_B = 2 \sum_j\sum_{w \in W_G} \frac{\sign(w)}{\delta(wX)} \biggl( &\prod_{\beta \in \Delta_G^+}\widehat{A}(\beta(wX)) e^{-\left(\alpha_j + \frac{\delta}{2}\right)(wX)} -\\
	    &- \prod_{\beta \in \Delta_G^+}\widehat{A}(\beta(wX|_\mathfrak{s})) e^{-(\kappa_j + \rho_H)(wX|_\mathfrak{s})}\biggr)\cdot \prod\limits_{\beta \in \Delta^+_G}\frac{-1}{\beta(X)}\bigg|_{X = 0}.
\end{aligned}
\end{equation*}

\item
	Consider two paths of $G$-equivariant Dirac operators $D^\lambda$ and $B^{\lambda,\mu}$ connecting the Dirac operator $D$ on $\Gamma(S)$ and the odd-signature operator $B$ on $\Omega^{ev}(M)$ induced by the Levi-Civita connection to their reductive counterparts $\widetilde{D}$ and $\widetilde{B}$ in the terminology of \cite{Goette1999} and \cite{Goette2009}. Using Frobenius reciprocity and Peter-Weyl theorem we write:
\[
	\Gamma(S) = \overline{\bigoplus\limits_{\gamma \in \widehat{G}} V^\gamma \otimes \Hom_H(V^\gamma, S)},
\]
\[\Omega^{ev}(M) = \overline{\bigoplus\limits_{\gamma \in \widehat{G}} V^\gamma \otimes \Hom_H(V^\gamma,S\otimes S)}.\]

For each summand above, we may write 
\begin{align}
	\label{eq:specflowD}
	D^\lambda|_{V^\gamma \otimes \Hom_H(V^\gamma, S)} &= \id_{V^\gamma} \otimes {}^\gamma D^\lambda, \\
	\label{eq:specflowB}
	B^{\lambda,\mu}|_{V^\gamma \otimes \Hom_H(V^\gamma,S\otimes S)} &= \id|_{V^\gamma} \otimes {}^\gamma{B}^{\lambda,\mu}.
\end{align}

The explicit formulas for ${}^\gamma D^\lambda$ and ${}^\gamma{B}^{\lambda,\mu}$ are given in the section \ref{secondterm}.

Then, $J$ terms are the spectral flow terms given as:
\begin{align}
	\label{eq:JD}
	J_D &= \sum\limits_{\gamma \in \widehat{G}} \chi_G^\gamma \cdot (\eta({}^\gamma D) - (\eta+h)( {}^\gamma\widetilde{D})),\\
	\label{eq:JB}
J_B &= \sum\limits_{\gamma \in \widehat{G}} \chi_G^\gamma \cdot (\eta({}^\gamma B) - (\eta+h)({}^\gamma \widetilde{B})).
\end{align}
\end{itemize}

Here, $h$ denotes the dimension of the kernel of the corresponding operator. 
\subsection{Computing first terms}
\label{firstterms}

We now consider the case of $G = SU(3)$, $H = S^1$ and define Cartan subalgebras as follows:
\[
\begin{aligned}
	\mathfrak{t} &= \{i\diag(x_1,x_2,x_3)\ | \ x_1 + x_2 + x_3 = 0\},\\
	\mathfrak{s} &= \{i\diag(kt,lt,-(k+l)t)\ | \ t \in \mathbb{R}\}. 
\end{aligned}
\]

Let $L_j \in i\mathfrak{t}^*$ be given by $L_j(i\diag(x_1,x_2,x_3)) = ix_j.$

The Weyl group $W_{SU(3)}$ is the symmetric group $S_3$. We pick the Weyl chamber $P_{SU(3)}= \{x_1 > x_2 > x_3\}$. With this choice the positive roots are:
\begin{equation*}
	\beta_1 = L_1 - L_2, \beta_2 = L_2 - L_3, \beta_3 = L_1 - L_3.
\end{equation*}

The half-sum of the weights is $\rho_G = \frac{1}{2}(\beta_1 + \beta_2 + \beta_3) = \beta_3.$ For $S^1$ we pick the Weyl chamber $P_{S^1} = \{t>0\}$. Since $H = S^1$, we have that $\rho_H = 0$. 

We now recall the orientation convention used in \cite{Goette2009}:

Let $\beta_1,\ldots, \beta_\ell$ be the positive roots of $\mathfrak{g}$, then one can choose a complex structure $I$ and a complex basis $z_1,\ldots, z_\ell$ such that $\ad|_{\mathfrak{t}\times \mathfrak{g}/\mathfrak{t}}$ takes 
\[
\ad_X = \begin{pmatrix}
\beta_1(X) &        & 0 \\
            & \ddots &   \\
0           &        & \beta_\ell(X)
\end{pmatrix}
\qquad \text{for all } X \in \mathfrak{t}.
\]
Then we declare the basis $z_1, I(z_1), \ldots z_\ell, I(z_\ell)$ of $\mathfrak{g}/\mathfrak{t}$ as a real vector space to be positive oriented.

In our case such choice of orientation corresponds to declaring the basis $(e_1,e_5,e_2,e_6,e_3,e_7)$ to be positive. Together with the orientation on $\mathfrak{m}$ this gives orientation on $\mathfrak{t}/\mathfrak{s} \cong \mathfrak{s}^\perp$:
\[
\mathfrak{m} = \mathfrak{g}/\mathfrak{t} \oplus (\mathfrak{t}/\mathfrak{s}).
\]
It now follows that the positive unit vector on $\mathfrak{s}^\perp$ is 
\[
	E = e_4.
\]

Thus,
\begin{equation*}
	\delta = \frac{1}{3}\left((2l+k)L_1 - (2k+l)L_2 + (k-l)L_3\right).
\end{equation*}

\begin{equation*}
 	\delta(x_1,x_2,-x_1-x_2) = i(lx_1 - kx_2).
\end{equation*}

Then the first term for the $\eta$ invariant of the Dirac operator can be expressed as:
\begin{equation*}
	I_D = 2 \sum_{w \in W_G} \frac{\sign(w)}{\delta(wX)} \left( \prod_{\beta \in \Delta_G^+}\widehat{A}(\beta(wX))\widehat{A}(\delta(wX))e^{- \frac{\delta}{2}(wX)} - 
\prod_{\beta \in \Delta_G^+}\widehat{A}(\beta(wX|_\mathfrak{s}))\right)\cdot \prod\limits_{\beta \in \Delta^+_G}\frac{-1}{\beta(X)}\bigg|_{X = 0}.
\end{equation*}

Let $\{\kappa_j\}$ be the weights of the representation of $\hat{\widetilde{\pi}}$ and $\{\alpha_j\}$ be their lifts. Then the first term of the $\eta$ invariant of the odd signature operator can be expressed as:

\begin{equation*}
\begin{aligned}
	I_B &= 2 \sum_j\sum_{w \in W_G} \frac{\sign(w)}{\delta(wX)} \biggl( \prod_{\beta \in \Delta_G^+}\widehat{A}(\beta(wX))\widehat{A}(\delta(wX))e^{-\left(\alpha_j - \frac{\delta}{2}\right)(wX)} -\\
	    &- \prod_{\beta \in \Delta_G^+}\widehat{A}(\beta(wX|_\mathfrak{s})) e^{-\kappa_j(wX|_\mathfrak{s})}\biggr)\cdot \prod\limits_{\beta \in \Delta^+_G}\frac{-1}{\beta(X)}\bigg|_{X = 0},
\end{aligned}
\end{equation*}

where the weights of $\hat{\widetilde{\pi}}$ are $(0,0, \pm i(k-l), \pm i(2k+l), \pm i(2l+k))$.

\begin{lemma}
	We have:
\begin{align*}
	I_D &= 	- \frac{81 kl(k+l)}{640(k^2+kl+l^2)^2} + \frac{1}{120} kl(k+l),\\
	I_B &= 8I_D + 2 =  2 -\frac{81 kl(k+l)}{80(k^2+kl+l^2)^2} + \frac{1}{15} kl(k+l).
\end{align*}
\label{lm:sum1}
\end{lemma}
The proof of this lemma requires a long computation, which is given in Section \ref{sect:comps}.

\subsection{Spectra of deformed Dirac operators}
\label{secondterm}

In this section we compute the spectral flow terms appearing in Goette's formulas for $\eta$ invariants. This calculation is similar to the computations of the $\eta$ invariants of the Berger space $SO(5)/SO(3)$ in \cite{Goette2004}. The main result of this section is:

\begin{lemma}
The spectral flow terms for the Dirac operator and odd signature operator are
\[
	J_D = 0,
\]
\[
	J_B = -2.
\]
\label{lm:sum2}
\end{lemma}
\begin{proof}%[Proof of Lemma \ref{lm:sum2}]
By \cite[Lemma 4]{Goette1999}, we have $\eta(D) = \eta(\widetilde{D})$, and by \cite[Lemma 1.17]{Goette1999} the kernel of $\widetilde{D}$ is trivial for $SU(3)/S^1$. Consequently, the spectral flow term from equation \eqref{eq:JD} is zero for the Dirac operator: 
\[
	J_D = 0.
\]

We now focus on the odd signature operator. Let $e_1,\ldots,e_7$ be the orthonormal basis of $\mathfrak{m}$ as in \ref{eq:basis}. Denote by $c_i$, $\hat{c}_i$ the Clifford multiplication by $e_i$ on the first and second factor of $\Lambda^{ev}\mathfrak{m} \cong S\otimes S$, respectively.

We define two maps $\widetilde{\ad}_{\mathfrak{m}}$ and $\widehat{\widetilde{\ad}}_{\mathfrak{m}}: \mathfrak{g} \to \End(\Lambda^{ev}\mathfrak{m})$ by 

\[
	\widetilde{\ad}_{\mathfrak{m}}(X) = \frac{1}{4}\sum_{i,j = 1}^7 \langle [X, e_i], e_j\rangle c_ic_j \text{ and } \widehat{\widetilde{\ad}}_{\mathfrak{m}}(X) = \frac{1}{4}\sum_{i,j = 1}^7 \langle [X, e_i], e_j\rangle \hat{c}_i\hat{c}_j,
\]
and set: 

\[
	\widetilde{\ad}_{\mathfrak{m}, i} = \widetilde{\ad}_{\mathfrak{m}}(e_i) \text{ and } \widehat{\widetilde{\ad}}_{\mathfrak{m},i} = \widehat{\widetilde{\ad}}_{\mathfrak{m}}(e_i).
\]
Then $\widetilde{\pi} = \widetilde{\ad}_{\mathfrak{m}}|_{\mathfrak{h}}$ and $\hat{\widetilde{\pi}} = \widehat{\widetilde{\ad}}_{\mathfrak{m}}|_{\mathfrak{h}}$ are the differentials of the representation of $H$ on the two factors of $S\otimes S$ that induce the bundle $\Lambda^{ev}\mathfrak{m}$.

Let $\gamma_i$ denote the action of $e_i$ on the dual of the representation space $V^\gamma$. Then the operators from equations \eqref{eq:specflowD} and \eqref{eq:specflowB} are defined in the following way:

\[
	{}^\gamma D^\lambda = \sum_{i=1}^7 c_i(\gamma_i + \lambda \widetilde{\ad}_{\mathfrak{m},i})
\]

\[
	{}^\gamma {B}^{\lambda,\mu} = \sum_{i=1}^7 c_i(\gamma_i + \lambda\widetilde{\ad}_{\mathfrak{m},i} + \mu\widehat{\widetilde{\ad}}_{\mathfrak{m},i}).
\]

Note that $D^{\tfrac{1}{2}} = D$ and $B^{\tfrac{1}{2},\tfrac{1}{2}}$ are respectively the Dirac operator and the odd signature operator associated to the Levi-Civita connection on $M$, while $D^{\frac{1}{3}} = \widetilde{D}$ and $B^{\frac{1}{3},0}$ are respectively the reductive Dirac operator and odd signature operators from the \cite{Goette1999}, \cite{Goette2009}.

Now, consider the one-parameter family $B^{\lambda,3\lambda-1}$ for $\lambda \in [\frac{1}{3}, \frac{1}{2}]$:

\[
	{}^\gamma {B}^{\lambda,3\lambda - 1} = {}^\gamma {\widetilde{B}} + \mu B_0.
\]

Where $B_0 = \sum_{i=1}^7 c_i\left(\frac{1}{3}\widetilde{\ad}_{\mathfrak{m},i} + \widehat{\widetilde{\ad}}_{\mathfrak{m},i}\right),$ and $\mu = 3\lambda-1$. 

The square of ${}^\gamma {\widetilde{B}}$ has been computed in  \cite[Lemma 1.17]{Goette1999}: 

\[
{}^\gamma {\widetilde{B}}^2 = ||\gamma + \rho_G||^2 - c_H^{\hat{\widetilde{\pi}}} - ||\rho_H||^2
\]

Since in our case $\rho_H = 0$, this simplifies to:
\[
{}^\gamma {\widetilde{B}}^2 = ||\gamma + \rho_G||^2 - c_H^{\hat{\widetilde{\pi}}}
\]

We now compute $c_H^{\hat{\widetilde{\pi}}}$. On the weight space $V_\mu$ of $S^1$ the Casimir operator is given as  $||\mu(h)||^2$ where $h$ is the unit generator of $\mathfrak{s}$ with respect to the norm induced from the embedding $\iota_{k,l}: S^1 \to SU(3)$. Hence, on the weight space $V_m$ the Casimir operator is $\frac{m^2}{2(k^2+kl+l^2)}$. In our case the weights of $\hat{\widetilde{\pi}}$ are $0, 0, \pm i(k-l), \pm i(2l+k), \pm i(2k+l)$. Thus, 
\[
	c_H^{\hat{\widetilde{\pi}}} \in \left\{0, \frac{(k-l)^2}{2(k^2+kl+l^2)}, \frac{(2k+l)^2}{2(k^2+kl+l^2)}, \frac{(2l+k)^2}{2(k^2+kl+l^2)} \right\}.
\] 
In particular,  
\[
	c_H^{\hat{\widetilde{\pi}}} \leqslant 2.
\]

Next we compute the $||\gamma + \rho_G||^2$ terms. Let $\gamma_{(p,q)}$ denote the irreducible representation of $SU(3)$ with the highest weight $pL_1 - qL_3$, then
\[
||\gamma_{(p,q)} + \rho_G||^2 = ||(p+1)L_1 - (q+1)L_3||^2 = \frac{2}{3}(p^2 + q^2 + pq) + 2(p+q) + 2.
\] 
Note that $||\gamma + \rho_G||^2 - c_H^{\hat{\widetilde{\pi}}}$ is always non-negative.

We now compute the eigenvalues of $B_0$ using the following model for $S$.
We identify $S$ with $\Lambda^* V$ where $V$ is 3-dimensional totally isotropic subspace with basis 
\[
	f_1 = \frac{1}{\sqrt2}(e_1 + ie_5), f_2 = \frac{1}{\sqrt2}(e_2 + ie_6), f_3 = \frac{1}{\sqrt2}(e_3 + ie_7).
\]

The Clifford multiplication is given as:
\begin{equation*}
	\begin{aligned}
	c_1 &= i(\varepsilon_1 + \iota_1)\\
	c_2 &= i(\varepsilon_2 + \iota_2)\\
	c_3 &= i(\varepsilon_3 + \iota_3)\\
	c_4 &= \pm i(-1)^{\deg}\\
	c_5 &= \iota_1 - \varepsilon_1\\
	c_6 &= \iota_2 - \varepsilon_2\\
	c_7 &= \iota_3 - \varepsilon_3.\\
	\end{aligned}
\end{equation*}

The choice of the sign for $c_4$ is given by the choice of orientation on $N_{k,l}$. We want the volume element $c_1c_5c_2c_6c_3c_7c_4$ to act as $\id$ on $S$. Since $c_1c_5c_2c_6c_3c_7$ acts as $i(-1)^{\deg}$, we have to choose $c_4 = -i(-1)^{\deg}$. 

In the basis \eqref{eq:basis}:

\begin{equation*}
	\begin{aligned}
	\widetilde{\ad}_{\mathfrak{m}, 1} &= \tfrac{1}{4}\left(-\sqrt2 c_2c_3 + \sqrt2 c_6c_7 - \tfrac{6(k+l)}{r}c_4c_5\right)\\
	\widetilde{\ad}_{\mathfrak{m}, 2} &= \tfrac{1}{4}\left(\sqrt2 c_1c_3 - \sqrt2 c_5c_7 + \tfrac{6k}{r}c_4c_6\right)\\
	\widetilde{\ad}_{\mathfrak{m}, 3} &= \tfrac{1}{4}\left(-\sqrt2 c_1c_2 + \sqrt2c_5c_6 + \tfrac{6l}{r}c_4c_7\right)\\
	\widetilde{\ad}_{\mathfrak{m}, 4} &= \tfrac{1}{4}\cdot \tfrac{6}{r}\left((k+l)c_1c_5 - kc_2c_6 -lc_3c_7\right)\\
	\widetilde{\ad}_{\mathfrak{m}, 5} &= \tfrac{1}{4}\left(\sqrt2c_2c_7 - \sqrt2c_3c_6 -\tfrac{6(k+l)}{r}c_1c_4\right)\\
	\widetilde{\ad}_{\mathfrak{m}, 6} &= \tfrac{1}{4}\left(-\sqrt2 c_1c_7 + \sqrt2c_3c_5 + \tfrac{6k}{r}c_2c_4\right)\\
	\widetilde{\ad}_{\mathfrak{m}, 7} &= \tfrac{1}{4}\left(\sqrt2c_1c_6 - \sqrt2 c_2c_5 + \tfrac{6l}{r}c_3c_4\right).\\
	\end{aligned}
\end{equation*}

Using the above model for spinors we compute the matrix $B_0$ acting on the space $S\otimes S$ and its eigenvalues in sympy. All of these computations can be found \href{https://github.com/Hypergu/Aloff-Wallach-spaces}{here}. The maximal absolute value of the eigenvalues of $B_0$ is $2\sqrt2$.

As we have seen before, 
\[
	({}^\gamma {\widetilde{B}})^2 \geqslant \frac{2}{3}(p^2 + q^2 + pq) + 2(p+q) \geqslant 2 = \lambda_{\max}^2\left(\frac{1}{2}B_0\right)
\]
for $(p,q) \neq (0,0)$. 

Consequently, the only irreducible representation $\gamma$ for which the sign of eigenvalues can change along the path ${}^\gamma {B}^{\lambda,3\lambda - 1}$ is the trivial representation $\gamma_0$. 

For the trivial representation, we explicitly compute the matrices ${}^{\gamma_0}{B}^{\frac{1}{3},0}$ and ${}^{\gamma_0}{B}^{\frac{1}{2},\frac{1}{2}}$ restricted to the $H$-invariant subspace $\Hom_H(V^{(0,0)},S\otimes S)$, which is 10-dimensional, with the basis given by:
\[
\left(
\begin{array}{r@{\,\otimes\,}l@{,\quad}r@{\,\otimes\,}l}
  1          & 1          & f_1f_2f_3  & 1,          \\[3pt]
  1          & f_1f_2f_3  & f_1f_2f_3  & f_1f_2f_3,  \\[3pt]
  f_1        & f_2f_3     & f_2f_3     & f_1,        \\[3pt]
  f_2        & f_1f_3     & f_1f_3     & f_2,        \\[3pt]
  f_3        & f_1f_2     & f_1f_2     & f_3
\end{array}
\right)
\]
\begin{remark}
	Note that this is only true for the general case $k \neq \pm l, l \neq \pm (k+l), k+l \neq \pm k$. In two exceptional cases $(1,-1,0)$ and $(1,1,2)$ the dimension of $\Hom_H(V^{(0,0)},S\otimes S)$ will be bigger and one would need to compute corresponding $\eta$ invariants separately. We discuss these cases in Section \ref{sect:except}.
\end{remark}
The corresponding $\eta$ invariants equal to:
\[
\eta({}^{\gamma_0}{B}^{\frac{1}{3},0}) = 2\sign(k) + 2\sign(l) + 2\sign(-k-l) = 2,\ h({}^{\gamma_0}{B}^{\frac{1}{3},0}) = 0.
\]
\[
\eta({}^{\gamma_0}{B}^{\frac{1}{2},\frac{1}{2}}) = 0,\ h({}^{\gamma_0}{B}^{\frac{1}{2},\frac{1}{2}}) = 2.
\]
The resulting difference is
\[
\eta({}^{\gamma_0}{B}^{\frac{1}{2},\frac{1}{2}}) - (\eta+h)({}^{\gamma_0}{B}^{\frac{1}{3},0}) = -2.
\]
Consequently,
\[
	J_B = -2.
\]
\end{proof}

\subsection{\texorpdfstring{Computing the $\bar{\nu}$ and $\xi$ invariants}{Computing the nu and xi invariants}}
Combining the results from the previous sections, we have:
\begin{proof}[Proof of Theorem \ref{thm:mainres}]
	Let $\alpha$ be some homogeneous $G_2$-structure inducing the normal metric and the given orientation. By the proof of Lemma \ref{lem:nu} we know, 
\begin{align*}
	\bar{\nu}(\alpha) &= \bar{\nu}(\varphi_1) = \bar{\nu}(\varphi_2),\\
	\xi(\alpha) &= \xi(\varphi_1) = \xi(\varphi_2).
 \end{align*}

By Proposition \ref{prop:homnu}, we know that 

\[
	\bar{\nu}(\alpha) = -24I_D + 3I_B - 24J_D + 3J_B.
\]

By Lemma \ref{lm:sum1} we have $-24I_D + 3I_B = 6$, and by \ref{lm:sum2} we have $-24J_D + 3J_B = -6.$

In total, we have \[
	\bar{\nu}(\varphi_1) = \bar{\nu}(\alpha) = 0.
\]

Using the formula for the $\xi$ invariant from Proposition \ref{prop:homxi}, equations \ref{eq:ID} and \ref{eq:IB} and Lemma \ref{lem:int}, we obtain:
\begin{align*}
	\xi(\varphi_1) &= \frac{45}{2}\left(I_B + J_B\right) - \frac{3}{8}\int_M p_1(\nabla^0)\wedge h(\nabla^0)\\
		     &= \frac{45}{2}\left(2 -\frac{81 kl(k+l)}{80(k^2+kl+l^2)^2} + \frac{1}{15} kl(k+l) - 2\right) + \frac{3}{8}\left(\frac{243kl(k+l)}{4(k^2+kl+l^2)^2}\right)\\
		     &= -\frac{729 kl(k+l)}{32(k^2+kl+l^2)^2} + \frac{3}{2}kl(k+l) + \frac{729 kl(k+l)}{32(k^2+kl+l^2)^2}\\
		     &= \frac{3}{2}kl(k+l).
\end{align*}
\end{proof}
\begin{proof}[Proof of Corollary \ref{cor:noneq}]
	If we have two pairs $(k,l)$ and $(k',l')$ such that $kl(k+l) \neq k'l'(k'+l')$, but $N_{k,l} \cong N_{k',l'}$, then the manifold $N_{k,l}$ admits $G_2$-structures induced by different homogeneous structures. By Theorem \ref{thm:mainres} we have $\xi(\varphi_{k,l}) = \frac{3}{2}kl(k+l) \neq \frac{3}{2}k'l'(k'+l') = \xi(\varphi_{k',l'})$ by the assumption, which implies that $G_2$-structures $\varphi_{k,l}$ and $\varphi_{k',l'}$ are not deformation equivalent.

	Consider example from Section \ref{sect:classif}: $N_{-4638861,582656} \cong N_{-2594149,5052965}$. We have:
	\[
		\xi(\varphi_{-4638861,582656}) = 16445032554770450432,
	\]
	while 
	\[
		\xi(\varphi_{-2594149,5052965}) = -48345771671661879296.
	\]
\end{proof}

\begin{remark}
	From Theorem \ref{thm:mainres} we can see that for a pair of Aloff--Wallach spaces $N_{k,l}$ and $N_{k',l'}$ one has:
    \[
    kl(k+l) - k'l'(k'+l') = 2^{5}\cdot 7^{\lambda(N)}\cdot 3\cdot q N,
    \]
    where $N= k^2+kl+l^2$ and $q \in \mathbb{Z}$. By the definition of $\lambda(N)$, we can write $N = 7^{1-\lambda(N)}\cdot r$ for $r \in \mathbb{N}$.

    Hence, 
    \[
    \xi(\varphi_{k,l}) - \xi(\varphi_{k',l'}) = \frac{3}{2}(kl(k+l) - k'l'(k'+l')) = 2^{4}\cdot 7\cdot 9\cdot q r.
    \]
    On the other hand, from \cite[Section 1.6]{Crowley2015}, we know that $\xi(\varphi_{k,l}) - \xi(\varphi_{k','l'}) = 14D(\varphi_{k,l},\varphi_{k',l'}),$ where $D : \pi_0(\mathcal{G}_2)\times \pi_0(\mathcal{G}_2) \to \mathbb{Z}$ is an affine difference defined in \cite[Section 3.1]{Crowley2015}

    So, in particular, we have: \[
	    D(\varphi_{k,l},\varphi_{k',l'}) = 72 \cdot qr = 72 \cdot q \cdot 7^{\lambda(N) - 1}N,
    \]
    where $N = |H^4(N_{k,l},\mathbb{Z})| = k^2 + kl + l^2$. It would be interesting to know if there are examples of manifolds where non-equivalent nearly-parallel $G_2$-structures can be \enquote{close} to each other. For example, are there examples of nearly-parallel $G_2$-structures $\varphi, \varphi'$ such that $D(\varphi,\varphi') = 1$?

    Additionally, note that $24\ |\ D(\varphi_{k,l},\varphi_{k',l'})$, which implies that $\nu(\varphi_{k,l}) - \nu(\varphi_{k',l'}) = 2D(\varphi_{k,l},\varphi_{k',l'}) = 0 \mod 48$, which is consistent with our computation of $\nu$.

\end{remark}

\subsection{\texorpdfstring{Computing the $s$ invariant}{Computing the s invariant}}
\label{sect:comp_s}
Using our results we also can compute the $s$ invariant for Aloff--Wallach spaces, which was computed in \cite{Kreck1993}.

\begin{align*}
	s(g) &= \frac{I_B + J_B}{2^5\cdot 7} + \frac{I_D + J_D}{2} - \frac{1}{2^7\cdot 7}\int_M p_1(\nabla^0)\wedge h(\nabla^0)\\
	     &= \frac{2 + 8I_D -2 + 2^4 \cdot 7 I_D}{2^5\cdot 7} + \frac{1}{2^7\cdot 7} \left(\frac{243kl(k+l)}{4(k^2+kl+l^2)^2}\right)\\
	     &= \frac{120I_D}{2^5 \cdot 7} + \frac{243kl(k+l)}{3584(k^2+kl+l^2)^2}\\
	     &= \frac{120}{2^5\cdot 7}\left(-\frac{81 kl(k+l)}{640(k^2+kl+l^2)^2} + \frac{1}{120} kl(k+l)\right) \frac{243kl(k+l)}{3584(k^2+kl+l^2)^2}\\
	     &= -\frac{243kl(k+l)}{3584(k^2+kl+l^2)^2} + \frac{1}{2^5\cdot 7}kl(k+l) + \frac{243kl(k+l)}{3584(k^2+kl+l^2)^2}\\
	     &= \frac{1}{2^5\cdot 7}kl(k+l).
\end{align*}

This is precisely the value computed in \cite{Kreck1993}.
\section{\texorpdfstring{$G_2$-structures associated with 3-Sasakian structures}{G2-structures associated with 3-Sasakian structures}}
\label{sect6}

Let $(M,g)$ be a 7-dimensional 3-Sasakian manifold. It is well-known that 3-Sasakian structure admits a 3-dimensional space of Killing spinors and hence a 2-sphere worth of nearly-parallel $G_2$-structures given by choosing a unit Killing spinor. We denote these structures by $\varphi_{ts}(x)$ for $x \in S^2$.

Moreover, from the data of 3-Sasakian structure one can construct a proper (in the sense that its space of Killing spinors is 1-dimensional) nearly-parallel $G_2$ structure called squashed nearly-parallel $G_2$ structure (cf. \cite{Friedrich1997}). We denote those by $\varphi_{sq}$.

We begin by recalling the definition of the 3-Sasakian structure.

\begin{definition}
	A 3-Sasakian structure on the manifold $(M,g)$ is a triple of vector fields $(V_1,V_2,V_3)$ such that the following is satisfied:
	\begin{enumerate}
		\item The vector $V_i$ defines Sasakian structure for each $i = 1,2,3$.
		\item The frame $(V_1,V_2,V_3)$ is orthonormal.
		\item For each permutation (i,j,k) of the sign $\delta:$ $\nabla_{V_i}V_j = (-1)^\delta V_k.$
		\item On the distribution orthogonal to $(V_1,V_2,V_3)$ the tensors $\phi_i = -\nabla V_i$ satisfy $\phi_i\phi_j = (-1)^\delta\phi_k.$
	\end{enumerate}
\end{definition}

A vector is called \textit{horizontal} if it is orthogonal to $V_i$ for $i = 1,2,3$. A vector is called \textit{vertical} if it lies in the span of $V_i$. 

For $t > 0$ define the canonical variation of the metric $g^t$:

$g^t(X,Y) = g(X,Y)$ if $X,Y$ are horizontal vector fields, and $g^t(V,W) = t^2g(V,W)$ if $V,W$ are vertical vectors. 

For $t = \frac{1}{\sqrt5}$ this metric is Einstein and admits proper nearly-parallel $G_2$ structure $\varphi_{sq}$.

\begin{lemma}
	The squashed nearly-parallel $G_2$ structure $\varphi_{sq}$ is homotopic to $\varphi_{ts}(x) \ \forall x$ along the path of $G_2$-structures inducing metrics with positive scalar curvature. 
	\label{lem:phisq}
\end{lemma}

\begin{proof}%[Proof of Lemma \ref{lem:phisq}]
	All of the $\varphi_{ts}(x)$ are homotopic since they correspond to a choice of a unit Killing spinor associated to the 3-Sasakian structure, which is connected.

	Fix an orthonormal frame of the horizontal distribution $X_1,X_2,X_3,X_4$ and define $Z_a := V_a/t$. Following \cite[Theorem 5.4]{Friedrich1997} we define the path of $G_2$ 3-forms in the following way:

	\[
		\varphi_t = F_1 + F_2,
	\]
	where 
	\[
		F_1 = Z_1 \wedge Z_2 \wedge Z_3,
	\]
	\[
		F_2 = \sum_a Z_a \wedge \omega_a, \text{ and } \omega_a = \frac{1}{2} \sum_i X_i \wedge \nabla_{X_i}V_a. 
	\]

	The form $\varphi_t$ induces precisely the metric $g_t$ and gives the path between $\varphi_1$ and $\varphi_{\frac{1}{\sqrt5}} = \varphi_{sq}$. We also note that the scalar curvature stays positive along this path: according to \cite[section 13.3.3]{Boyer2007} the scalar curvature of the metric $g^t$ is 
	\[
		s_t = 48 + \frac{6}{t^2} - 12t^2.
	\]
Which is positive for $t \in \left[\tfrac{1}{\sqrt5}, 1\right]$.

Note that $\varphi_1$ is not the one of nearly-parallel $G_2$-structures induces from 3-Sasakian structure. 

	According to \cite{Agricola2010} the $G_2$ structure $\varphi_1$ (called the canonical $G_2$ structure associated to 3-Sasakian structure) admits a spinor field $\zeta_0$, which generates Killing spinors by taking the Clifford product with the horizontal vectors $V_a$. In particular, the space of Killing spinors is generated by $V_a \cdot \zeta_0$ for $a = 1,2,3$. 

	But since $X\cdot \zeta_0 \perp \zeta_0$ for any vector field $X$, these spinors can be continuously rotated one into another via
	\[
		\zeta_t = \zeta_0\cos t + V_a \cdot \zeta_0\sin t.
	\]
	Hence, the corresponding $G_2$-structures are homotopic, completing the proof. 
\end{proof}

\begin{remark}
	By the h-principle for coclosed $G_2$-structures \cite[Theorem 1.8]{Crowley2015}, it follows that $\varphi_{ts}$ and $\varphi_{sq}$ can be connected via the path of coclosed $G_2$-structures. Note that the path constructed in the Lemma \ref{lem:phisq} does not necessary stay in the space of coclosed $G_2$-structures, since the coclosed condition is not generally preserved under linear operations on the corresponding spinors. 
\end{remark}

\begin{ex} In particular, we can compute the $\bar{\nu}$ invariant of the squashed metric on $S^7$:
	\label{ex:sphere}
\begin{equation*}
	\label{eq:nuS7}
	\bar{\nu}(\varphi_{sq}(S^7)) = \bar{\nu}(\varphi_{std}(S^7)) = 1.
\end{equation*}

The last equality is due to \cite[Example 1.9]{Crowley2025}. 
\end{ex}
\section{\texorpdfstring{Pontryagin classes of homogeneous $G_2$ manifolds}{Pontryagin classes of homogeneous G2 manifolds}}
\label{sect7}

In this section we would like to gather some results regarding the first Pontryagin class of homogeneous nearly-parallel $G_2$ manifolds. In particular, as we will see all of the known examples have first Pontryagin class rationally trivial, which forces $\pi_0\bar{\mathcal{G}}_2 = \mathbb{Z}$ and implies that $\xi$ is a complete invariant of $\pi_0\bar{\mathcal{G}}_2$ and there exists an intrinsic formula for the $\xi$ invariant, which is computable in some special cases.

The main statement of this section is:

\begin{prop}
	Let $M$ be a 7-dimensional homogeneous space admitting the homogeneous nearly parallel $G_2$ structure. Then the first Pontryagin class $p_1(M)$ is a torsion class. 
	\label{prop:p1}
\end{prop}
\begin{proof}%[Proof of Proposition \ref{prop:p1}]
	The homogeneous nearly parallel $G_2$ manifolds were classified by \cite[Theorem 7.2]{Friedrich1997}. The only proper homogeneous examples are the squashed 7-sphere, Aloff--Wallach spaces and the Berger space. 

	Other homogeneous nearly parallel $G_2$ manifolds are not proper, hence they are Sasaki-Einstein and according to \cite[Proposition 2.2]{LeBrun2025} any Sasaki-Einstein 7-manifold has $p_1(M)$ a torsion class. 
	
	Next we compute $p_1$ explicitly for the proper cases. 

\begin{itemize}
	\item In the case of $S^7$, $p_1$ is trivially zero.
	\item According to \cite{Kruggel1997}, the fourth cohomology group of $N_{k,l}$ is $H^4(N_k,l) \cong \mathbb{Z}_{(k^2 + kl + l^2)}$, and the first Pontryagin class $p_1(N_{k,l})$ is zero. 
	\item Let $\xi_{m,n}$ be the vector bundle of the rank $4$ over $S^4$ with the Euler class $e(\xi_{m,n}) = n$ and first Pontryagin class $p_1(\xi_{m,n}) = 2(n+2m)$. Let $M_{m,n}$ be the corresponding $S^3$ bundle over $S^4$. Then, according to \cite{Goette2004} the Berger space $SO(5)/SO(3)$ is diffeomorphic to $M_{\mp 1, \pm 10}$. 
		The algebraic topological invariants of $M_{m,n}$ are computed in \cite{Crowley2003}: 
	\[H^4(M_{m,n}) \cong \mathbb{Z}_{n}, \ p_1(M_{m,n}) = 4m \in \mathbb{Z}_{n}.\] 
	Hence, \[p_1(SO(5)/SO(3)) = -4 \in \mathbb{Z}_{10}.\]
\end{itemize}

\end{proof}

In fact, as we can see, all of the known examples (including nonhomogeneous ones obtained from Sasaki-Einstein or 3-Sasakian structures) of nearly parallel $G_2$ manifolds have torsion $p_1$ class (although, sometimes for trivial reasons). In particular, $\xi$ invariant is a complete invariant of deformation classes of $G_2$-structures in these cases.
\appendix

\section{Computations}
In the appendix we carry out explicit computations required to prove Lemma \ref{lm:sum1}. In the first part we discuss the problem of lifting weights of $\hat{\widetilde{\pi}}$ from $i\mathfrak{s}$ to $i\mathfrak{t}$, which is required to compute $\eta$ of the odd signature operator $B$. In the second part we explicitly compute the sum $-24I_D + 3I_B$. 

\label{appx}
\subsection{\texorpdfstring{Computing lifts of the weights of $\hat{\widetilde{\pi}}$}{Computing lifts of the weights}}
\label{sect:lifts}
As we have seen in the section \ref{sect5} the weights of $\widetilde{\pi}$ are $0, \pm i(k-l), \pm i(2l+k), \pm i(2k+l)$.

In this section we compute lifts of the weights of $\widetilde{\pi}$. 

By the lift of the weight $\kappa \in i\mathfrak{s}^*$ to the weight $\alpha \in i\mathfrak{t}^*$ we understand the unique weight such that ${\alpha}|_\mathfrak{s} = \kappa + \rho_H$ and $-i(\alpha-\delta)(E)< 0 \leqslant -i\alpha(E)$ for $\delta, E$ as in \ref{sect5}.

 Obviously, by this definition $0$ lifts to $0$.

 Let $r:= ||(2l+k,-2k-l,k-l)|| = \sqrt{6(k^2 + kl + l^2)}.$ Then $\delta(E) = \frac{r}{3}$, $E = \frac{1}{r}(2l+k,-2k-l,k-l)$.

Note that under our assumptions on the integers $k,l$, we have $k+l \geqslant 3$. And, in particular, we have $k^2 + l^2 + (k+l)^2 > 3\cdot \max (k,l,k+l) = 3(k+l).$ Hence we have that
\[
	\delta(E) > \max\left(\frac{3(k+l)}{r}, \frac{3k}{r}, \frac{3l}{r}\right).
\]
\begin{itemize}
	\item consider the weight $i(k-l)$. Then the lift should be of the form $\beta_1 + m\delta$ for some $m \in \mathbb{Z}$. 
	\[
	0 \leqslant -i\beta_1(E) = \frac{3(k+l)}{r} < \delta(E).
	\]
	hence the lift is
	\[
		\alpha = \beta_1.
	\]
	
	\item consider the weight $-i(2l+k)$. Then the lift should be of the form $-\beta_2 + m\delta$ for some $m \in \mathbb{Z}$. 
	\[
	0 \leqslant i\beta_2(E)  = \frac{3k}{r} < \delta(E).
	\]
	hence, the lift is
	\[
		\alpha = -\beta_2.
	\]
	\item consider the weight $i(2k+l)$. Then the lift should be of the form $\beta_3 + m\delta$ for some $m \in \mathbb{Z}$. 
	\[
	0 \leqslant -i\beta_3(E) = \frac{3l}{r} < \delta(E).
	\]
	hence the lift is
	\[
		\alpha = \beta_3.
	\]
\end{itemize}

\begin{remark}
	\label{rmk:neglifts}
	Now, assume that weight $\kappa$ lifts to $\alpha$, i.e.
\[
	0 \leqslant -i\alpha(E) < -i\delta(E). 
\]

Consider the weight $-\kappa$, then 
\[
	 -i\delta(E) > -i(\delta(E) - \alpha(E)) > 0.
\]

Thus, the weight $-\kappa$ lifts to $\delta - \alpha$. In particular $-i(k-l)$ lifts to $\delta - \beta_1$, $-i(2k+l)$ lifts to $\delta - \beta_3$, and $i(2l+k)$ lifts to $\delta + \beta_2$. 
\end{remark}

\subsection{\texorpdfstring{Computing the $I$ terms}{Computing the I terms}}
\label{sect:comps}
In this section we give the proof of the Lemma \ref{lm:sum1} by examining the terms constituting expressions for $I_D$ and $I_B$.

The expressions for the $I_D$ and $I_B$ have an apparent singularity of the fourth order, which cancels out. Since the expressions are analytic it is enough to compute fourth power terms in the Taylor series of corresponding functions. 

We denote by $\widetilde{I}_D$ the fourth power term of 
\[
\left( \prod_{\beta \in \Delta_G^+}\hat{A}(\beta(wX))\widehat{A}(\delta(wX)) e^{- \frac{\delta}{2}(wX)} - 
\prod_{\beta \in \Delta_G^+}\hat{A}(\beta(wX|_\mathfrak{s}))\right)
\]
and by $\widetilde{I}_B$ the fourth power term of

\[
 \sum_j\biggl( \prod_{\beta \in \Delta_G^+}\widehat{A}(\beta(wX))\widehat{A}(\delta(wX)) e^{-\left(\alpha_j + \frac{\delta}{2}\right)(wX)} -
\prod_{\beta \in \Delta_G^+}\widehat{A}(\beta(wX|_\mathfrak{s})) e^{-\kappa_j(wX|_\mathfrak{s})}\biggr)
\]

In particular 
\[
	I_D = \sum_{w \in W_{SU(3)}} \frac{\sign(w)}{\delta(wX)}\widetilde{I}_D(wX)\prod\limits_{\beta \in \Delta^+_G}\frac{-1}{\beta(X)}\bigg|_{X = 0},
\]
and
\[
	I_B = \sum_{w \in W_{SU(3)}} \frac{\sign(w)}{\delta(wX)}\widetilde{I}_B(wX)\prod\limits_{\beta \in \Delta^+_G}\frac{-1}{\beta(X)}\bigg|_{X = 0}.
\]

Note that
\[\widehat{A}(z) = \frac{z/2}{\sinh{z/2}} = 1 - \frac{1}{24}z^2 + \frac{7}{5760}z^4 + \ldots \]

Denote $a = -\frac{1}{24}, b = \frac{7}{5760}$. 

Since we have three positive roots in the case of $SU(3)/S^1$, the terms that we need to compute have the following expression:

\[ \widehat{A}(\beta_1)\widehat{A}(\beta_2)\widehat{A}(\beta_3)\widehat{A}(\delta)e^{-\left(\alpha-\frac{\delta}{2}\right)} - \widehat{A}(\widetilde{\beta_1})\widehat{A}(\widetilde{\beta_2})\widehat{A}(\widetilde{\beta_3})e^{-\kappa}. 
\]

Fourth power term is given as: 

\begin{equation}
\begin{aligned}
&b\left(\beta_1^4 + \beta_2^4 + \beta_3^4 - \widetilde{\beta}_1^4 - \widetilde{\beta}_2^4 - \widetilde{\beta}_3^4 + \delta^4\right) + \\
	+& a^2\left(\beta_1^2\beta_2^2 + \beta_2^2\beta_3^2 + \beta_3^2\beta_1^2 - \widetilde{\beta}_1^2\widetilde{\beta}_2^2 - \widetilde{\beta}_2^2\widetilde{\beta}_3^2 - \widetilde{\beta}_3^2\widetilde{\beta}_1^2\right)+\\
	 +& a^2\left(\beta_1^2 + \beta_2^2 + \beta_3^2\right)\delta^2  
+ \frac{a}{2}\left(\beta_1^2 + \beta_2^2 + \beta_3^2 +\delta^2\right)\left(\alpha - \frac{\delta}{2}\right)^2 -\\
-& \frac{a}{2}\left(\widetilde{\beta}_1^2 + \widetilde{\beta}_2^2 + \widetilde{\beta}_3^2\right)\kappa^2 
+ \frac{1}{24}\left(\alpha - \frac{\delta}{2}\right)^4 
- \frac{1}{24}\kappa^4 = \\
=& b\left(\beta_1^4 + \beta_2^4 + \beta_3^4 - \widetilde{\beta}_1^4 - \widetilde{\beta}_2^4 - \widetilde{\beta}_3^4 + \delta^4\right) +\\
+& a^2\left(\beta_1^2\beta_2^2 + \beta_2^2\beta_3^2 + \beta_3^2\beta_1^2 - \widetilde{\beta}_1^2\widetilde{\beta}_2^2 - \widetilde{\beta}_2^2\widetilde{\beta}_3^2 - \widetilde{\beta}_3^2\widetilde{\beta}_1^2\right) + \\
+& a^2\left(\beta_1^2 + \beta_2^2 + \beta_3^2\right)\delta^2  
+ \frac{a}{2}\left(\beta_1^2 + \beta_2^2 + \beta_3^2 +\delta^2\right)\left(\alpha^2 - \alpha\delta + \frac{\delta^2}{4}\right) +\\
+& \frac{1}{24}\left(\alpha^4 - 2\alpha^3 \delta + \frac{3}{2}\alpha^2\delta^2 -\frac{1}{2}\alpha\delta^3 + \frac{1}{16}\delta^4\right)-\\ 
-& \frac{a}{2}\left(\widetilde{\beta}_1^2 + \widetilde{\beta}_2^2 + \widetilde{\beta}_3^2\right)\kappa^2 - \frac{1}{24}\kappa^4.
\end{aligned}
\label{eq:4term}
\end{equation}

Denote 
\[
	U := \beta_1^4 + \beta_2^4 + \beta_3^4 - \widetilde{\beta}_1^4 - \widetilde{\beta}_2^4 - \widetilde{\beta}_3^4,
\]
\[
	V := \beta_1^2\beta_2^2 + \beta_2^2\beta_3^2 + \beta_3^2\beta_1^2 - \widetilde{\beta}_1^2\widetilde{\beta}_2^2 - \widetilde{\beta}_2^2\widetilde{\beta}_3^2 - \widetilde{\beta}_3^2\widetilde{\beta}_1^2.
\]

In particular, since for the standard Dirac operator $\kappa = 0$ and $\alpha = 0$, we have:

\begin{equation*}
	\widetilde{I}_D = bU + a^2V +\left(a^2 + \frac{a}{8}\right)\left(\beta_1^2 + \beta_2^2 + \beta_3^2\right)\delta^2 + \left(b + \frac{a}{8} + \frac{1}{24\cdot 16}\right)\delta^4.
\end{equation*}

Now, consider $\widetilde{I}_B$. Recall that $\hat{\widetilde{\pi}}$ has eight weights. The two zero weights contribute two terms $\widetilde{I}_D$, since they lift to zeroes. The remaining non-zero weights occur in pairs $(\kappa_i, -\kappa_i)$. By remark \ref{rmk:neglifts}, if $\kappa_i$ lifts to $\alpha$, $-\kappa_i$ lifts to $\delta - \alpha$.

Observe that these weights appear in \eqref{eq:4term}, only as $\kappa_i^{\text{even}}$ and $(\alpha - \delta/2)^{\text{even}}$. Since $(\delta - \alpha - \delta/2) = -(\alpha-\delta/2)$, the even powers ensure that $\kappa_i$ and $-\kappa_i$ make identical contributions to $\widetilde{I}_B$. Recall that nontrivial weights of $\hat{\widetilde{\pi}}$ lift to $(\pm \beta_i, \delta\mp \beta_i)$. Taking the sum over the weights of $\hat{\widetilde{\pi}}$ and substituting the weights in the formula \eqref{eq:4term} we obtain:

\begin{align*}
	\widetilde{I}_B &= 2 \widetilde{I}_D + 2\biggl(3bU + 3a^2V + 3b\delta^4 + 3a^2\delta^2\left(\beta_1^2 + \beta_2^2 +\beta_3^2\right)+\\
	&+ \frac{a}{2}\left(\beta_1^2 + \beta_2^2 +\beta_3^2+\delta^2\right)\left(\left(\beta_1^2 + \beta_2^2 +\beta_3^2\right) - \left(\beta_1 - \beta_2 + \beta_3\right)\delta + \frac{3\delta^2}{4}\right)+\\
	&+ \frac{1}{24}\biggr(\left(\beta_1^4 + \beta_2^4 + \beta_3^4\right) - 2\left(\beta_1^3 - \beta_2^3 + \beta_3^3\right)\delta + \frac{3}{2}\left(\beta_1^2 + \beta_2^2 + \beta_3^2\right)\delta^2 - \\
	&- \frac{1}{2}\left(\beta_1 - \beta_2 + \beta_3\right)\delta^3 + \frac{3}{16}\delta^4\biggr) -\\
	&- \frac{a}{2}\left(\widetilde{\beta}_1^2 + \widetilde{\beta}_2^2 + \widetilde{\beta}_3^2\right)\left(\widetilde{\beta}_1^2 + \widetilde{\beta}_2^2 + \widetilde{\beta}_3^2\right) - \frac{1}{24}\left(\left(\widetilde{\beta}_1^4 + \widetilde{\beta}_2^4 + \widetilde{\beta}_3^2\right)\right) = \\
	&=2\widetilde{I}_D + 2\biggl( \left(3b +\frac{a}{2} + \frac{1}{24}\right)U + \left(3a^2 + a\right)V + \left(3b +\frac{3a}{8} + \frac{3}{16\cdot 24}\right)\delta^4 +\\
	&+ \left(3a^2 + \frac{7a}{8} + \frac{3}{48}\right)\delta^2\left(\beta_1^2+\beta_2^2+\beta_3^2\right) \\
	&- \frac{a}{2}\left(\beta_1^2 + \beta_2^2 + \beta_3^2\right)\left(\beta_1 - \beta_2 + \beta_3\right)\delta\\
	&- \frac{2}{24}\left(\beta_1^3 - \beta_2^3 + \beta_3^3\right)\delta + \left(-\frac{1}{48} -\frac{a}{2}\right)\left(\beta_1 - \beta_2 + \beta_3\right)\delta^3\biggr)\\
	&= 8\widetilde{I}_D + 2\biggl(\left(\frac{a}{2} + \frac{1}{24}\right)U + aV \\
	&+ \left(\frac{4a}{8} + \frac{3}{48}\right)\delta^2\left(\beta_1^2+\beta_2^2+\beta_3^2\right) \\
	&- \frac{a}{2}\left(\beta_1^2 + \beta_2^2 + \beta_3^2\right)\left(\beta_1 - \beta_2 + \beta_3\right)\delta\\
	&- \frac{1}{12}\left(\beta_1^3 - \beta_2^3 + \beta_3^3\right)\delta + \left(-\frac{1}{48} -\frac{a}{2}\right)\left(\beta_1 - \beta_2 + \beta_3\right)\delta^3\biggr)\\
	&= 8\widetilde{I}_D + \frac{1}{24}(U-2V) + \frac{1}{12}\delta^2\left(\beta_1^2+\beta_2^2+\beta_3^2\right)\\
	&+\frac{1}{24}\left(\beta_1^2 + \beta_2^2 + \beta_3^2\right)\left(\beta_1 - \beta_2 + \beta_3\right)\delta -\frac{1}{6}\left(\beta_1^3 - \beta_2^3 + \beta_3^3\right)\delta.
\end{align*}

\begin{lemma} We have that 
	\[
		U = 2V.
	\]
	\label{lm:u=2v}
\end{lemma}
\begin{proof}%[Proof of Lemma \ref{lm:u=2v}]

Let 
\[
	z =  \frac{kx_1 + lx_2 +(k+l)(x_1+x_2)}{2(k^2+kl+l^2)},
\]
so that
\[
X|_{\mathfrak{s}} = z\cdot i(k,l,-k-l).
\]
Then,
\begin{align*}
	U &= \beta_1^4 + \beta_2^4 + \beta_3^4 - \widetilde{\beta}_1^4 - \widetilde{\beta}_2^4 - \widetilde{\beta}_3^4 =\\
	&= (x_1 - x_2)^4 + (2x_2 + x_1)^4 + (2x_1+x_2)^4 - z^4((k-l)^4 + (2l+k)^4 + (2k+l)^4) =\\
	&= (18x_1^4 + 36x_1^3x_2 + 54x_1^2x_2^2 + 36x_1x_2^3 + 18x_2^4) - z^4(18k^4 + 36k^3l + 54k^2l^2 + 36kl^3 + 18l^4).
\end{align*}

\begin{align*}
	(x_1 - x_2)^4 + (2x_2 + x_1)^4 + (2x_1+x_2)^4 &= 18x_1^4 + 36x_1^3x_2 + 54x_1^2x_2^2 + 36x_1x_2^3 + 18x_2^4,\\
	(k-l)^4 + (2l+k)^4 + (2k+l)^4) &= 18k^4 + 36k^3l + 54k^2l^2 + 36kl^3 + 18l^4.
\end{align*}

\begin{align*}
	V =& \beta_1^2\beta_2^2 + \beta_2^2\beta_3^2 + \beta_3^2\beta_1^2 - \widetilde{\beta}_1^2\widetilde{\beta}_2^2 - \widetilde{\beta}_2^2\widetilde{\beta}_3^2 - \widetilde{\beta}_3^2\widetilde{\beta}_1^2 = \\
	=& (x_1-x_2)^2(2x_2+x_1)^2 + (2x_2+x_1)^2(2x_1+x_2)^2 + (2x_1+x_2)^2(x_1-x_2)^2 -\\
	-& z^4((k-l)^2(2l+k)^2 + (2l+k)^2(2k+l)^2 + (2k+l)^2(k-l)^2) =\\
	=& (9x_1^4 + 18x_1^3x_2 + 27x_1^2x_2^2 + 18x_1x_2^3 + 9x_2^4) 
	- z^4(9k^4 + 18k^3l + 27k^2l^2 + 18kl^3 + 9l^4).
\end{align*}

Thus, we can see that $U = 2V$. 
\end{proof}
Hence, we have that 
\begin{align*}
	\widetilde{I}_B = 8\widetilde{I}_D  + \frac{1}{12}\delta^2\left(\beta_1^2+\beta_2^2+\beta_3^2\right) &+\frac{1}{24}\left(\beta_1^2 + \beta_2^2 + \beta_3^2\right)\left(\beta_1 - \beta_2 + \beta_3\right)\delta\\
														      &-\frac{1}{6}\left(\beta_1^3 - \beta_2^3 + \beta_3^3\right)\delta.
\end{align*}
Hence, 

\begin{align*}
	\frac{1}{\delta}\left(\widetilde{I}_B - 8\widetilde{I}_D\right) = \frac{1}{12}(\beta_1^2 + \beta_2^2 +\beta_3^2)\delta &+ \frac{1}{24}(\beta_1^2 + \beta_2^2 + \beta_3^2)(\beta_1 - \beta_2 + \beta_3) - \\
	&- \frac{1}{6}(\beta_1^3 - \beta_2^3 + \beta_3^3).
\end{align*}

It is easy to see via direct computation that after symmetrization over $W_{SU(3)} = S_3$ only the $(\beta_1^3 - \beta_2^3 + \beta_3^3)$ term gives nonzero value $6$, so:

\begin{align*}
	I_B - 8I_D &= 2\cdot \sum_{w \in W_G} \sign(w)\frac{\widetilde{I}_B(wX)-24\widetilde{I}_D(wX)}{\delta(wX)} \cdot \prod_{\beta \in \Delta_+} \frac{-1}{\beta(X)} =\\
		      &= 2 \cdot \left(\frac{-6}{6}\right)(-1)^3 = 2.
\end{align*}

In particular, we have

\begin{equation}
    I_B = 8I_D + 2.
\end{equation}

So, it is sufficient to compute only $I_D$.

Now, we compute the value of $I_D$, which is also enough to compute $I_B$. Recall
\begin{align*}
	\widetilde{I}_D &= (b + a^2/2)U +\left(a^2 + \frac{a}{8}\right)\left(\beta_1^2 + \beta_2^2 + \beta_3^2\right)\delta^2 + \left(b + \frac{a}{8} + \frac{1}{24\cdot 16}\right)\delta^4,\\
	\widetilde{I}_D/\delta &= (b + a^2/2)U/ +\left(a^2 + \frac{a}{8}\right)\left(\beta_1^2 + \beta_2^2 + \beta_3^2\right)\delta + \left(b + \frac{a}{8} + \frac{1}{24\cdot 16}\right)\delta^3.
\end{align*}

\begin{enumerate}
	\item A long computation gives 
\begin{align*}
	U &= \frac{18}{16(k^2+kl+l^2)^4}\cdot \bigg((24 k^{6} l^{2} + 72 k^{5} l^{3} + 135 k^{4} l^{4} + 150 k^{3} l^{5} + 117 k^{2} l^{6} + 54 k l^{7} + 15 l^{8})x_1^4 \\
	  &+ (- 48 k^{7} l - 120 k^{6} l^{2} - 180 k^{5} l^{3} - 120 k^{4} l^{4} - 12 k^{3} l^{5} + 72 k^{2} l^{6} + 60 k l^{7} + 24 l^{8})x_1^3x_2 \\
	  &+ (24 k^{8} + 24 k^{7} l - 30 k^{6} l^{2} - 156 k^{5} l^{3} - 210 k^{4} l^{4} - 156 k^{3} l^{5} - 30 k^{2} l^{6} + 24 k l^{7} + 24 l^{8})x_1^2x_2^2 \\
	  &+ (24 k^{8} + 60 k^{7} l + 72 k^{6} l^{2} - 12 k^{5} l^{3} - 120 k^{4} l^{4} - 180 k^{3} l^{5} - 120 k^{2} l^{6} - 48 k l^{7})x_1x_2^3\\
	  &+ (15 k^{8} + 54 k^{7} l + 117 k^{6} l^{2} + 150 k^{5} l^{3} + 135 k^{4} l^{4} + 72 k^{3} l^{5} + 24 k^{2} l^{6})x_2^4\bigg).
\end{align*}

One can check that 
\begin{align*}
	U/\delta &= \frac{18}{16(k^2+kl+l^2)^4} \cdot \bigg( ( 24 k^{6} l + 72 k^{5} l^{2} + 135 k^{4} l^{3} + 150 k^{3} l^{4} + 117 k^{2} l^{5} + 54 k l^{6} + 15 l^{7})x_1^3 \\
		 &+ (- 24 k^{7} - 48 k^{6} l - 45 k^{5} l^{2} + 30 k^{4} l^{3} + 105 k^{3} l^{4} + 126 k^{2} l^{5} + 75 k l^{6} + 24 l^{7})x_1^2x_2 \\
		 &+ (- 24 k^{7} - 75 k^{6} l - 126 k^{5} l^{2} - 105 k^{4} l^{3} - 30 k^{3} l^{4} + 45 k^{2} l^{5} + 48 k l^{6} + 24 l^{7})x_1x_2^2 \\
		 &+ (- 15 k^{7} - 54 k^{6} l - 117 k^{5} l^{2} - 150 k^{4} l^{3} - 135 k^{3} l^{4} - 72 k^{2} l^{5} - 24 k l^{6})x_2^3 \bigg).
\end{align*}

After symmetrization over the Weyl group the $U$-term gives us:

\[
	\frac{18\cdot 27 (k^{6} l + 3 k^{5} l^{2} + 5 k^{4} l^{3} + 5 k^{3} l^{4} + 3 k^{2} l^{5} + k l^{6})}{16(k^2+kl+l^2)^4} = \frac{18\cdot 27\cdot kl(k+l)}{16(k^2+kl+l^2)^2}
\]

\item The second term is 
	\begin{align*}
		\left(\beta_1^2 + \beta_2^2 + \beta_3^2\right)\delta &= 6 l x_{1}^{3}- 6 (k-l) x_{1}^{2} x_{2} - 6 (k-l) x_{1} x_{2}^{2} - 6 k x_{2}^{3},
	\end{align*}
	Which after symmetrization over the Weyl group second term gives zero. 

\item The last term is 
	\[\delta^3 = l^{3} x_{1}^{3} - 3 k l^{2} x_{1}^{2} x_{2} + 3 k^{2} l x_{1} x_{2}^{2} - k^{3} x_{2}^{3},\]
	which gives $-3kl(k+l)$ after symmetrization.
\end{enumerate}

So, symmetrizing $\widetilde{I}_D$ gives us 
\[
	\frac{1}{480} \cdot\frac{18\cdot 27\cdot kl(k+l)}{16(k^2+kl+l^2)^2} - \frac{1}{240} kl(k+l)
\]
Since, 
\[
	I_D = \sum_{w \in W_{SU(3)}} \frac{\sign(w)}{\delta(wX)}\widetilde{I}_D(wX)\prod\limits_{\beta \in \Delta^+_G}\frac{-1}{\beta(X)}\bigg|_{X = 0},
\]
we have that
\begin{align}
	I_D &= - \left(\frac{81 kl(k+l)}{640(k^2+kl+l^2)^2} - \frac{1}{120} kl(k+l)\right)\notag \\
	&= - \frac{81 kl(k+l)}{640(k^2+kl+l^2)^2} + \frac{1}{120} kl(k+l).
	\label{eq:ID}
\end{align}

Then, 
\begin{equation}
	I_B = 2 -\frac{81 kl(k+l)}{80(k^2+kl+l^2)^2} + \frac{1}{15} kl(k+l).
	\label{eq:IB}
\end{equation}
\subsection{\texorpdfstring{Computing $\int_M p_1 \wedge h$}{Computing the integral}}
\label{sect:int}
In this section we compute the value of $\int_M p_1\left(\nabla^{0}\right) \wedge h\left(\nabla^0\right)$ for Aloff--Wallach spaces.

\begin{lemma}
	\label{lem:int}
	Under the orientations choices from the Section \ref{sect4}, we have
\begin{equation}
    \int_{N_{k,l}} p_1\left(\nabla^0\right) \wedge h\left(\nabla^0\right) = -\frac{243kl(k+l)}{4(k^2+kl+l^2)^2}.
\end{equation}
\end{lemma}
\begin{proof}
    
\begin{enumerate}
\item    
First, we compute the first Pontryagin form with respect to the reductive connection. For the reductive connection $\nabla^0$, we have that 
\[
	R\left(\nabla^0\right)(V,W) = - \pi_{*[V,W]_{\mathfrak{h}}}.
\]

\begin{align*}
	R(\nabla^0) &= -\pi_{*[e_1,e_5]_\mathfrak{h}} e^1 \wedge e^5 - \pi_{*[e_2,e_6]|_\mathfrak{h}} - \pi_{*[e_3,e_7]_|\mathfrak{h}}e^3 \wedge e^7\\
		    &= -\frac{(k-l)}{2(k^2+kl+l^2)}\pi_{*h}e^1\wedge e^5 - \frac{(2l+k)}{2(k^2+kl+l^2)} e^2 \wedge e^6 \\
		    &+ \frac{(2k+l)}{2(k^2+kl+l^2)}e^3 \wedge e^7.
\end{align*}

Then
\begin{align*}
	p_1\left(\nabla^0\right) &= -\frac{1}{8\pi^2} \tr\left(R\left(\nabla^0\right) \wedge R\left(\nabla^0\right)\right) \\
		      &= -\frac{2\tr(\pi_{*h}^2)}{8\pi^2 \cdot 4(k^2+kl+l^2)^2} \bigl((k-l)(2l+k)e^{1526} \\
		      &-(2l+k)(2k+l)e^{2637} - (2k+l)(k-l)e^{3715}\bigr)
\end{align*}	
\[
	\tr(\pi_{*h}) = 2\left((k-l)^2 + (2k+l)^2 + (2l+k)^2\right) = 12\left(k^2+kl+l^2\right).
\]
It is easy to see that the only terms in $R\left(\nabla^0\right)$ are $e^1\wedge e^5$, $e^2 \wedge e^6$ and $e^3\wedge e^7$.

So, we have:
\begin{align*}
	p_1(e_1,e_5,e_2,e_6) &= -\frac{1}{4\pi^2\cdot (k^2+kl+l^2)} \bigl((k-l)(2l+k)e^{1526}\\
			     &-(2l+k)(2k+l)e^{2637} - (2k+l)(k-l)e^{3715}\bigr)
\end{align*}

\item Second, we find the 3-form $h$ such that $dh = p_1$.

The $H$-invariant space of 3-forms is generated by 
\begin{align*}
	e^{154},\ &e^{264}, e^{374}, \Rel\left((e^1+ie^5)\wedge(e^2 + ie^6)\wedge(e^3 + ie^7)\right), \text{ and }\\
	&\Img\left((e^1+ie^5)\wedge(e^2 + ie^6)\wedge(e^3 + ie^7)\right).
\end{align*}

Denote \[
	R := \Rel\left((e^1+ie^5)\wedge(e^2 + ie^6)\wedge(e^3 + ie^7)\right) = e^{123}-e^{167}+e^{257}-e^{356}.
\]
Using the formula for the differential $d\varphi(V_1,V_2) = - \varphi([V_1,V_2]_{\mathfrak{m}})$ and commutator relations from Section \ref{AWspaces}, we find that

\begin{align*}
	de^1 &= -\frac{1}{\sqrt2}\left(e^{67}-e^{23}\right) + \frac{3(k+l)}{r}e^{45},\\
	de^2 &= -\frac{1}{\sqrt2}\left(e^{13}-e^{57}\right) - \frac{3k}{r}e^{46},\\
	de^3 &= -\frac{1}{\sqrt2}\left(e^{56}-e^{12}\right) - \frac{3l}{r}e^{47},\\
	de^4 &= -\frac{3(k+l)}{r}e^{15}+\frac{3k}{r}e^{26}+\frac{3l}{r}e^{37},\\
	de^5 &= -\frac{1}{\sqrt2}\left(e^{27}-e^{36}\right) + \frac{3(k+l)}{r}e^{14},\\
	de^6 &= -\frac{1}{\sqrt2}\left(e^{35}-e^{17}\right) - \frac{3k}{r}e^{24},\\
	de^7 &= -\frac{1}{\sqrt2}\left(e^{16}-e^{25}\right) - \frac{3l}{r}e^{34}.
\end{align*}

Then using the Leibniz rule for the differential one can find that 
\begin{equation*}
    h\left(\nabla^0\right) = \frac{3r}{8\pi^2(k^2+kl+l^2)}\left((k+l)e^{145}-ke^{246}-le^{347}\right) - \frac{(k^2+kl+l^2)}{2\sqrt2}R.
\end{equation*}

Note that $p_1 \wedge R = 0$, so the last term doesn't contribute to the quantity that we are interested in.

\item Finally, we have

\begin{align*}
	p_1 \wedge h &= -\frac{9r}{32\pi^4 (k^2+kl+l^2)^2} \bigl(-(k+l)(2l+k)(2k+l)e^{2637145} +k(2k+l)(k-l)e^{3715246} \\
		     &-l(k-l)(2l+k)e^{1526347}\bigr)\\
		     &= -\frac{3r}{32\pi^4 (k^2+kl+l^2)^2} \bigl((k+l)(2l+k)(2k+l) - k(2k+l)(k-l) \\
		     &+l(k-l)(2l+k)\bigr)e^{1234567}\\
		     &= -\frac{81r}{32\pi^4 (k^2+kl+l^2)^2}\cdot kl(k+l)e^{1234567}
\end{align*}

Under the choice of the Killing form $\tr(XY)$ the volume of $SU(3)$ is $16\pi^5\sqrt{3}$, hence, the volume of $N_{k,l}$ is 
\[
\Vol(N_{k,l}) = \Vol(SU(3))/\Vol(S^1_{k,l}) = \frac{16\pi^5\sqrt{3}}{2\pi \sqrt{2(k^2+kl+l^2)}} = \frac{8\pi^4 \sqrt{3}}{\sqrt{2(k^2+kl+l^2)}}.
\]

Under the chosen orientation for the $N_{k,l}$ the volume form is given as $e^{1526374} = e^{1234567}.$

Hence,

\begin{equation*}
    \int p_1 \wedge h = -\frac{243kl(k+l)}{4(k^2+kl+l^2)^2}. \qedhere
\end{equation*}
\end{enumerate}
\end{proof}
\subsection{Exceptional cases}
\label{sect:except}

In this section we discuss exceptional cases $N_{1,1}$ and $N_{1,0}$. As we have seen in section \ref{sect:lifts}, the general inequality $k^2+kl+l^2 > 3\cdot \max(k,l,k+l)$ is not satisfied in these cases. Additionally, some of the weights of $S\otimes S$ degenerate in these cases making the $H$-invariant subspace of $S\otimes S$ larger than in general case. 

\subsubsection{\texorpdfstring{$N_{1,1}$}{N11}}

We compute the lifts of the weights of $\widetilde{\pi}$ in the case $k = l = 1$, which is not covered by the general argument. 

In this case $r = \sqrt{6(k^2+kl+l^2)} = 3\sqrt2$ and  $E = \frac{i}{3\sqrt2}(3,-3,0) = \frac{i}{\sqrt2}(1,-1,0)$. $\delta(X) = i(x_1-x_2) = \beta_1$. $\delta(E) = i\sqrt2$.

The weights $i(k-l) = 0, i(2k+l) = 3i, i(2l+k) = 3i$. 

The weight $i(2l+k)$ lifts to $\delta + \beta_2$ and $i(2k+l)$ to $\beta_3$ which coincide, since $\delta = \beta_1$. 

Weight $i(k-l)$ lifts to $0$, but we can write it as $\delta - \beta_1$. 

So the lifts of the weights are:

\begin{align*}
	\delta - \beta_1, & \delta - \beta_1;\\
	-\beta_2, & \delta + \beta_2;\\
	\beta_3, & \delta - \beta_3.
\end{align*}

In particular, the formula for $I_B$ is the same as in the general case considered in section \ref{sect:comps}.

\subsubsection{\texorpdfstring{$N_{1,0}$}{N10}}

Now, consider the case $k = 1, l = 0$.

In this case $r = \sqrt{6(k^2+kl+l^2)} = \sqrt6$, $E = \frac{i}{\sqrt6}(1,-2,1)$, and $\delta(E) = \frac{2}{\sqrt6} = \frac{\sqrt6}{3}$. 

Weights $i(k-l) = i$, $i(2l+k) = i$, $i(2k+l) = 2i$. 

It is easy to see that weight $i(k-l)$ lifts to $\beta_1 - \delta$, $-i(2l+k)$ lifts to $-\beta_2 - \delta$, $i(2k+l)$ lifts to $\beta_3$ while $-i(2k+l)$ lifts to $-\beta_3$, since $\beta_3(E) = 0$. So the lifts are:

\begin{align*}
	\beta_1 - \delta,\ & 2\delta - \beta_1;\\
	-\beta_2 - \delta,\ & 2\delta + \beta_2;\\
	\beta_3,\ & - \beta_3.
\end{align*}
Note that in this case $\beta_1 - \delta = 2\delta + \beta_2$, since $i(k-l) = i(2l+k) = i$. 

We now revisit the computation of $I_B$ from section \ref{sect:comps}. The only difference comes from the terms that include $\beta_i$ in odd powers. For $N_{1,0}$ we have:

\begin{align*}
	\widetilde{I}_B  =&2\widetilde{I}_D + 2\biggl( \left(3b +\frac{a}{2} + \frac{1}{24}\right)U + \left(3a^2 + a\right)V + \left(3b +\frac{3a}{8} + \frac{3}{16\cdot 24}\right)\delta^4 +\\
	&+ \left(\frac{4a}{8} + \frac{3}{48}\right)\delta^2\left(\beta_1^2+\beta_2^2+\beta_3^2\right) \\
	&- \frac{a}{2}\left(\beta_1^2 + \beta_2^2 + \beta_3^2\right)\left((\beta_1-\delta) + (- \beta_2 - \delta)\right)\delta\\
	&- \frac{2}{24}\left((\beta_1-\delta)^3 + (- \beta_2 - \delta)^3\right)\delta + \left(-\frac{1}{48} -\frac{a}{2}\right)\left((\beta_1-\delta) + (- \beta_2 - \delta)\right)\delta^3\biggr).
\end{align*}
And 
\begin{align*}
	\widetilde{I}_B/\delta = 8\widetilde{I}_D/\delta + \frac{1}{12}(\beta_1^2 + \beta_2^2 +\beta_3^2)\delta &+ \frac{1}{24}(\beta_1^2 + \beta_2^2 + \beta_3^2)((\beta_1-\delta) + (- \beta_2 - \delta)) - \\
	&- \frac{1}{6}((\beta_1-\delta)^3 + (- \beta_2 - \delta)^3).
\end{align*}
In this case, after symmetrizing over the Weyl group, we can see that all terms give $0$ and 
\[
	I_B = 8I_D.
\]

One would also need to compute the eigenvalues of ${}^{\gamma_0}{B}^{\frac{1}{2},\frac{1}{2}}$ and ${}^{\gamma_0}{B}^{\frac{1}{3},0}$ separately for each of the exceptional Aloff--Wallach spaces, since $(S\otimes S)^H$ is larger for them than in the general case. However, we will use a shortcut:

Note that in general we have $I_B = 8I_D + c$, where c is some constant ($c = 0$ in the case of $N_{1,0}$ and $c = 2$ otherwise). Moreover, the computations of $I_D$ and $\int_M p_1\wedge h$ from sections \ref{sect:comps} and \ref{sect:int} are valid for all values of $k,l$. Then, using the computation of the $s$ invariant from section \ref{sect:comp_s} and the result of \cite{Kreck1993} that $s(N_{k,l}) = \frac{kl(k+l)}{2^5\cdot 7}$, one can deduce that $c + J_B = 0$, which then can be applied to the cases of exceptional Aloff--Wallach spaces to verify the formulas for $\bar{\nu}$ and $\xi$.

\section*{Acknowledgments}
The author would like to thank his advisor C. LeBrun for guidance and support, S. Goette for answering questions about his papers and D. Crowley and J. Nordstr\"om for their comments and suggestions. 
\newpage
\begingroup
\setlength{\emergencystretch}{.8em}
\printbibliography
\endgroup

\end{document}